\documentclass[12pt]{amsart}
\usepackage{lineno}

\usepackage{graphicx} 
\usepackage[export]{adjustbox}
\usepackage{amssymb,amsmath}

\usepackage{mathtools}

\usepackage{comment}

\newcommand{\sss}{\mathfrak{s}}
\newcommand{\sib}{sib}

\newcommand{\lab} {lab}
\newcommand{\col} {col}
\newcommand{\hth} {ht}
\newcommand{\sign} {sign}
\newcommand{\spin} {spin}
\newcommand{\type}{tp}

\newcommand{\wR}{\widehat{R}}
\newcommand{\wcol} {\widehat{col}}
\newcommand{\wsign} {\widehat{sign}}
\newcommand{\wspin} {\widehat{spin}}
\newcommand{\whth} {\widehat{ht}}

\newcommand{\semb}{$*$-embedding }
\newcommand{\sembs}{$*$-embeddings }

\newcommand{\N}{\mathbb{N}}
\newcommand{\Z}{\mathbb{Z}}
\newcommand{\SC}{\mathcal{C}}
\newcommand{\SD}{\mathcal{D}}

\newcommand{\SP}{\mathcal{P}}

\newcommand{\SR}{\mathcal{R}}
\newcommand{\SSS}{\mathcal{S}}
\newcommand{\Sp}{\mathcal{S}^p}
\newcommand{\ST}{\mathcal T}

\def\Power #1 { \powerset(#1) }

\newtheorem{definition}{{\bf Definition}}[section]
\newtheorem{theorem}{{\bf Theorem}}
\newtheorem*{theorem*}{{\bf Theorem}}

\newtheorem{corollary}[definition]{{\bf Corollary}}
\newtheorem{proposition}[definition]{\noindent {\bf Proposition}}
\newtheorem{lemma}[definition]{\noindent {\bf Lemma}}

\newtheorem{observation}[definition]{\noindent {\bf Observation}}

\newcommand{\restrict}[2]{#1\mspace{-3mu}\mathbin{\restriction}\mspace{-2mu} #2}

\begin{document}

\title[Bonato-Tardif, Thomass\'e, and Tyomkyn Conjectures]{An example of Tateno disproving conjectures of Bonato-Tardif, Thomasse, and Tyomkyn}

\author [D.Abdi]{Davoud Abdi Kalow}\address{
Mathematics \& Statistics Department, University of Calgary, Calgary, T2N1N4, Alberta, Canada  T2N1N4}
\email{davoud.abdikalow@ucalgary.ca}
\author[C.Laflamme]{Claude Laflamme*} 
\address{Mathematics \& Statistics Department, University of Calgary, Calgary, Alberta, Canada T2N 1N4}
\email{laflamme@ucalgary.ca} 
\thanks{*Supported by NSERC of Canada} 
\author[A.Tateno]{Atsushi Tateno} 
\author[R.Woodrow]{Robert Woodrow*} 
\address{Mathematics \& Statistics Department, University of Calgary, Calgary, Alberta, Canada T2N 1N4}
\email{woodrow@ucalgary.ca }  

\dedicatory{Dedicated to  our parents.}

\date{2022-05-25}


\keywords{trees, siblings}
\subjclass[2000]{ Relational Structures, Partially ordered sets and lattices (06A, 06B)}

\begin{abstract}  
In his 2008 thesis, Tateno claimed a counterexample to the Bonato-Tardif conjecture regarding the number of equimorphy classes of trees. 
In this paper we revisit Tateno's unpublished ideas to provide a rigorous exposition, constructing locally finite trees having an  arbitrary finite number of equimorphy classes;  an adaptation provides partial orders with a similar conclusion.
At the same time these examples also disprove conjectures by Thomass\'e and Tyomkyn.
\end{abstract}  

\maketitle

\section{Introduction} 
  
Two structures $R$ and $S$ are \emph{equimorphic},  denoted by $R \approx S$, when each embeds in the other; we may also say that one is 
a \emph{sibling} of the other. If $R$ is finite, there is just one sibling (up to isomorphy). The famous Cantor-Bernstein-Schroeder Theorem states that this is also the case for structures in a language with pure equality: if there is an injection from one set to another and vice-versa, then there is a bijection between these two sets. The same situation occurs in other structures such as vectors spaces, where embeddings are linear injective maps. But generally one cannot expect equimorphic structures to be necessarily  isomorphic: the rational numbers, considered as a linear order, has up to isomorphism continuum many siblings. It is thus a natural problem to understand the siblings of a given   structure, and as a first approach to count those siblings (up to isomorphy). 

Thus, let $\sib(R)$ be the number of siblings of $R$, these siblings being  counted up to isomorphism. Thomass\'e conjectured that $\sib(R)=1$, $\aleph_0$ or $2^{\aleph_0}$ for countable relational structures made of at most countably many relations  (Conjecture 2 in  \cite{thomasse}).   There is a special case of interest, namely whether  $\sib(R)=1$ or infinite for a relational structure of any cardinality. This was unsettled even in the case of locally finite trees, and is connected to the Bonato-Tardif conjecture which asserts that either all trees equimorphic to a given arbitrary tree $T$ are isomorphic, or else there are infinitely many pairwise non-isomoprohic trees equimorphic to $T$, also called the \emph{Tree Alternative Conjecture} (see \cite{bonato-tardif, bonato-al, tyomkyn}). 
Note that, as a binary relational structure, a ray has infinitely many siblings (add an arbitrary finite disconnected path), but a ray has no non-isomorphic sibling in the category of trees. The subtle connection between these conjectures is through the following observation by Hahn, Pouzet and Woodrow \cite{hahn}:  every sibling of a tree $T$ (as a binary relational structure, or graph) is a tree if and  only if $T\oplus 1$ (the graph obtained by adding an isolated vertex  to $T$) is not a sibling of $T$ (more generally, note that every sibling of a connected graph is connected, just in case $G\oplus 1$ is not a sibling). Hence, for a tree $T$ not equimorphic to $T\oplus 1$, the Bonato-Tardif conjecture (in the category of trees) and the special case of Thomass\'e's conjecture (in the category of relational structures) are equivalent.  

Bonato and Tardif \cite{bonato-tardif} proved their conjecture for rayless trees, and this was extended to rayless graphs by Bonato, Bruhn, Diestel and Spr\"ussel \cite{bonato-al}. It was also verified for the case of rooted trees by Tyomkyn \cite{tyomkyn}, and in addition made some progress towards the conjecture for locally finite trees. Tyomkyn made a first conjecture that if there exists a non-surjective embedding of a locally finite tree $T$, then $\sib(T) $ is infinite unless $T$ is a ray, a conjecture which immediately implies the Bonato-Tardif conjecture for locally finite trees. Tyomkyn further conjectured an apparently weaker version that if there exists a non-surjective embedding of a locally finite tree $T$, then $T$ has at least one non-isomorphic sibling unless $T$ is a ray. Laflamme, Pouzet and Sauer \cite{lps} later proved the  Bonato-Tardif conjecture for scattered trees, that is those trees not containing a subdivision of the binary tree. In fact they proved the result under the slightly more general notion of a stable tree. This is based on extensions of results of Polat and Sabidussi \cite{polat-sabidussi}, Halin \cite{halin, halin2, halin3}, and Tits \cite{tits} on automorphisms of trees.  Moreover they proved Tyomkyn's first conjecture holds for locally finite scattered trees.  Hamann \cite{hamann}, making use of the monoid of embeddings, deduced  the Bonato-Tardif conjecture for trees not satisfying two specific structural properties of that monoid. More recently, Abdi \cite{abdi} showed that a tree satisfying that first property is stable,  and therefore  the  Bonato-Tardif conjecture also holds in that case. 

In a parallel direction, Thomass\'e's conjecture has been fully verified for countable chains, and its special case also verified for all chains by Laflamme, Pouzet and Woodrow  \cite{lpw}, paving the way toward partial orders. A first step was made for direct sums of chains by Abdi \cite{abdi}, and after Hahn, Pouzet and Woodrow \cite{hahn} proved the special case of the conjecture in the special case of cographs, Abdi \cite{abdi} extended this result to closely related NE-free posets.   Another supporting indication came with the special case of the conjecture for a countable $\aleph_0$-categorical relational structure, proved by Laflamme, Pouzet, Sauer and Woodrow \cite{lpsw}, and extended by  Braunfeld et al \cite{braunfeld-laskowski}.

In this paper we revisit Tateno's unpublished ideas to provide a rigorous exposition, constructing locally finite trees having an  arbitrary finite number of equimorphy classes.  At the same time these examples disprove the above conjectures of Thomass\'e and Tyomkyn.

\begin{theorem}\label{thm:main1} 
For each non-zero $\sss \in \N$, there is a locally finite tree $\ST=\ST_{\sss}$ with exactly $\sss$ siblings (up to isomorphy), considered either as relational structures or trees. Moreover, for $\sss=1$, the tree is not a ray yet has a non-surjective embedding. \\
Thus the conjectures of Bonato-Tardif, Thomass\'e, and Tyomkyn regarding the sibling number of trees and relational structures are all false. 
\end{theorem} 

This result has been a long time coming, but counterexamples had already been produced by Pouzet (see \cite{hahn}, \cite{lpw}) in the categories of directed graphs and simple graphs with loops:

\begin{figure}[ht]
\includegraphics[width=.8\textwidth]{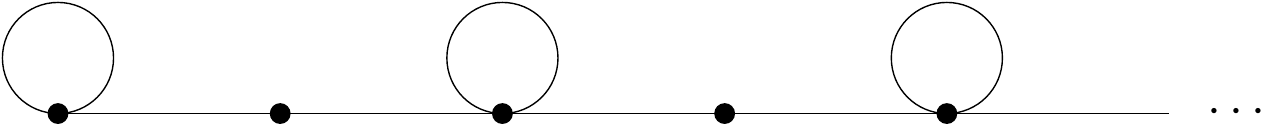}
\caption{In the category of {\em connected graphs with loops}, the above structure has exactly 2 siblings.}
\end{figure}

\begin{figure}[ht]
\includegraphics[width=.6\textwidth]{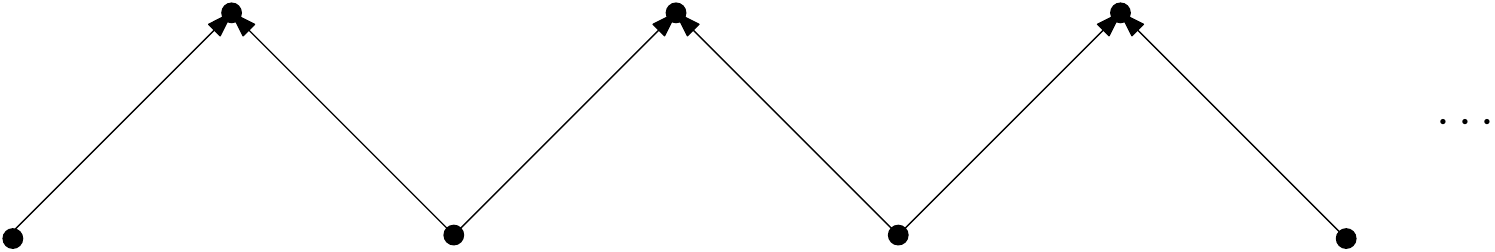}
\caption{Similarly in the category of {\em connected posets}, the one way infinite fence has exactly two siblings.}
\end{figure}

Indeed the above trees can be adapted to also provide partial orders with an arbitrary finite number of siblings.

\begin{theorem}\label{thm:main2}
For each non-zero $\sss \in \N$, there is a partial order $\SP$ with exactly $\sss$ siblings (up to isomorphy).
\end{theorem} 

We warmly thank Maurice Pouzet for bringing these problems to our attention, and for his generosity sharing his insight and expertise over the years on the subject. 
We also thank Mykhaylo Tyomkyn for making us aware of the claimed counterexample,  and Atsuhi Tateno for making his mansucript available to us and becoming a co-author. 

\section{Construction of the Locally Finite Trees} 
The strategy is  to build locally finite trees as a finite set of pairwise non-isomorphic siblings $\langle \ST_s: s<\sss \rangle$ for any fixed non-zero $\sss \in \N$, such that any sibling of $\ST=\ST_0$ (as a binary relational structure) is isomorphic to some $\ST_s$.  This yields a locally finite tree $\ST$ such that $\sib(\ST)=\sss$. The case $\sss=2$ will already disprove the conjectures of Bonato-Tardif (and hence Tyomkyn's first conjecture) and of Thomass\'e since we will show that $T \oplus 1$ does not embed in $T$. The special case $\sss=1$ will disprove Tyomkyn's second conjecture. The case $\sss>2$ is only for additional information, showing that any finite number can be the sibling number of some locally finite tree.
These trees will later be adapted to provide similar results for partial orders. 

The construction of each $\ST_s$ will be done in a similar manner as a countable union of trees, coding the countably many potential siblings within the trees  along the way. Moreover  $\ST_s \setminus \phi(\ST_s)$ will be finite for every embedding $\phi$, and hence all such differences will be captured after a finite stage of the construction, allowing to eventually show that all siblings have been accounted for. To facilitate the exposition, the construction will initially make use of several non graph properties  (labels, type assignments, sign and spin for example), but all will be eventually replaced by genuine graph properties. This means for example that embeddings will first be assumed to preserve the non graph properties, and to be clear will note those as \sembs; eventually it will be shown that these non graph properties are actually preserved by graph embeddings (or just ``embeddings'') alone; the most delicate case being through the Main Lemma \ref{similarityemb}.

\subsection{Rooted tree $\SR=(R, r)$}\label{rootedtree}

We begin by constructing a rooted tree $\SR=(R, r)$ that will be used repeatedly throughout the construction. We will first develop local properties of that tree, and then later extend them to each $\ST_s$.

The tree $\SR$ is built using a labelling on the vertices $\lab:R\rightarrow \N$ to guide the construction.  We first declare $\lab(r) = 0$, and then we construct the tree inductively under the following rules:
\begin{itemize}
\item If $\lab(v) = 0$, then $v$ has exactly two neighbours of label 1. 
\item If $\lab(v) \neq 0$, then $v$ has exactly three neighbours labelled $\lab(v)-1, \lab(v), \lab(v)+1$.
\end{itemize}

\begin{figure}[ht]
\includegraphics[width=\textwidth]{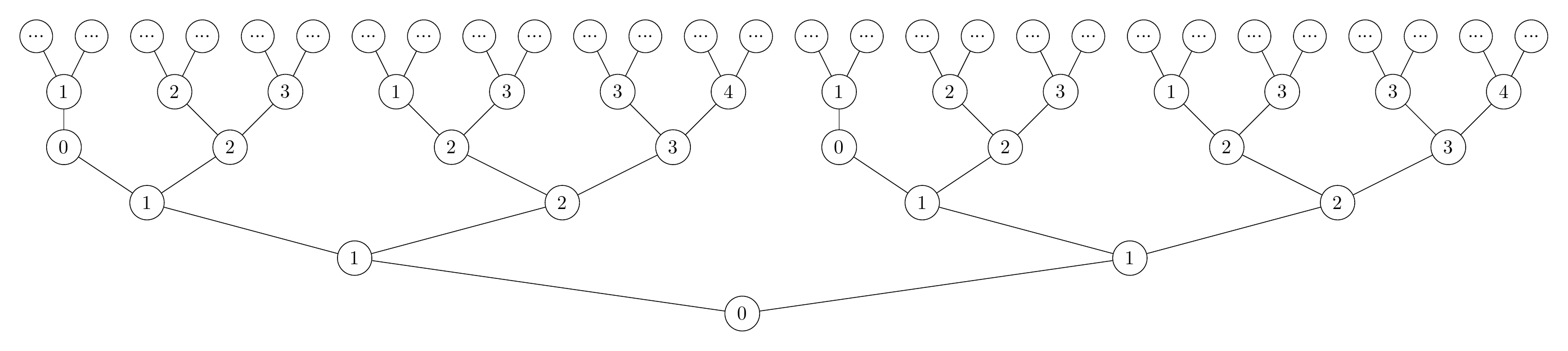}\label{rootedr}
\caption{Rooted Tree $\SR=(R,r)$ and the vertex labelling}
\end{figure}

We denote by $R_0$ the 0-labelled vertices of $\SR=(R,r)$ and often call them {\em tree vertices}. Note that these are exactly the vertices of degree 2 in $\SR$. Now that the tree has been constructed, one notices that the labelling of a vertex can be recovered  from the tree itself as the (graph)  distance to the nearest tree vertex (vertex of degree 2).

\begin{observation}\label{reconstructlabels}
For any vertex $v \in R$,  $\lab(v)$ is the distance to the nearest vertex of degree 2. 
\end{observation}

\begin{proof}
Tree vertices are exactly those of degree 2, all other vertices of label $\ell>0$ have  degree $3$. Now any vertex of label $\ell$ has a path of length $\ell$ with decreasing labels to a tree vertex, hence its distance to the nearest tree vertex is at most $\ell$. On the other hand labels decrease by at most $1$, therefore no tree vertex can be any closer. 
\end{proof}

Thus the labels are a graph property and any (graph) embedding of $\SR$ preserves labels. Yet, we will be joining several copies of $\SR$ and adding vertices to the tree, so we want to ensure that these labels can be recovered from the eventual graph structure. We do so  by encoding these labelled vertices using finite trees (gadgets) rooted at those vertices as follows. First consider the bipartite graph $K_{1,m}=(u,V)$ and call $PK(n,m)$ the finite tree formed by connecting the vertex $u$ to the end leaf of a path of length $n$; then attach the root of a copy of $PK(2\ell+6,2)$ to any vertex $v$ with $\lab(v)=\ell$. 
\begin{figure}[ht]
\includegraphics[width=.4\textwidth]{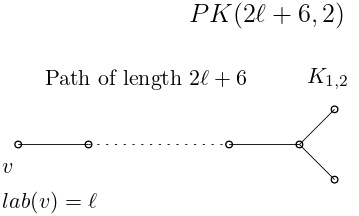}\label{pkn}
\caption{The gadget $PK(2\ell+6,2)$: finite rooted tree used to encode label $\ell$.}
\end{figure}

What is important here is that these gadgets are pairwise non-embeddable as {\em rooted trees} (mapping roots to roots) for different values of $\ell$, and thus play the graph theoretic role of the labels. From this point we want to make the label a graph property by attaching these finite gadgets to vertices.   Technically, this results in a tree extension $(R^a,r)$.   Note that any graph embedding $\phi$ of the resulting tree must take vertices in $R$ to vertices in $R$ and it must map the label gadget attached at a vertex $v$ of $R$ to the gadget at $\phi(v)$.  Since distinct gadgets do not embed in each other they must be the same.   So in fact $\phi$ must map $R^a$ onto itself, mapping the label gadgets to label gadgets.    Having noted this important property we abuse the notation using $(R,r)$ for $(R^a,r)$.

We will henceforth for notaional simplicity continue to use the labels $\lab(v)$ themselves, with the understanding that the labels are a graph property and preserved by embeddings. 

In anticipation of the construction of $\ST$ (and each $\ST_s$), we note that we will eventually associate a double ray to each tree vertex in a copy of $(R,r)$, and amalgamate the ray and the copy by identifying a single vertex of the ray and the tree vertex. The reason for $\ell +3$ above is simply that we will later use a small versions to code type assignments on those ray vertices.

Another important observation is that there is a (label preserving) embedding $\phi:\SR=(R,r) \rightarrow (R,v)$ sending $r$ to $v$ for any tree vertex $v$; moreover this is an automorphism since all embeddings of $\SR$ are surjective. This remains true even when we make the labels explicitly a graph property by adding the finite tree labels.

\begin{observation}\label{symmetry}
All embeddings of $\SR$ are surjective, and $(R,r)$ is  isomorphic to $(R,v)$ (as rooted trees) for any tree vertex $v$.
\end{observation}
In fact $(R,r)$ is so symmetric that one main point of the construction of $\ST$ is to insert obstructions to control its graph embeddings and as a result the number of siblings. 

The following notions of colour and height for tree vertices will be used in the inductive construction of each $\ST_s$. We call a pair of adjacent vertices in $\SR$ {\em consecutive} if they have the same label. Note that if a tree vertex $v \in R_0$ is different from $r$, then the path $P_{r,v}$ from $r$ to $v$ must contain at least one such a consecutive pair. 

\begin{definition}\label{colour}
For any tree vertex  $v \in R_0$, define $\col_v:R_0 \rightarrow \N$, the colour with respect to $v$,  by 
\[\col_v(u)=\lab(w), \mbox{ for } u \neq v,\] 
where $w$ is a vertex from the last consecutive pair in $P_{v,u}$. \\
For convenience let $\col_v(v)=0$. 
\end{definition}
Thus, since labels are preserved by graph embeddings, we have that $ \col_r(v)=\col_{\phi(r)}(\phi(v))$ for any graph embedding $\phi:R \rightarrow R$ and tree vertex $v$.
However more is true, and this kind of argument will play a central role throughout. 

\begin{lemma}\label{colpreserv}
For any  tree vertices $u, v \in R_0$, $\col_u(w)=\col_v(w)$ for all but finitely many $w \in R_0$.  \\
The only possible exceptions are tree vertices on paths starting from a vertex in $P_{u,v}$ and with strictly decreasing labels. 
\end{lemma}

\begin{proof}
We may assume that $w \notin P_{u,v}$ since those are part of the exceptional vertices.

\begin{figure}[ht]
\includegraphics[width=.6\textwidth]{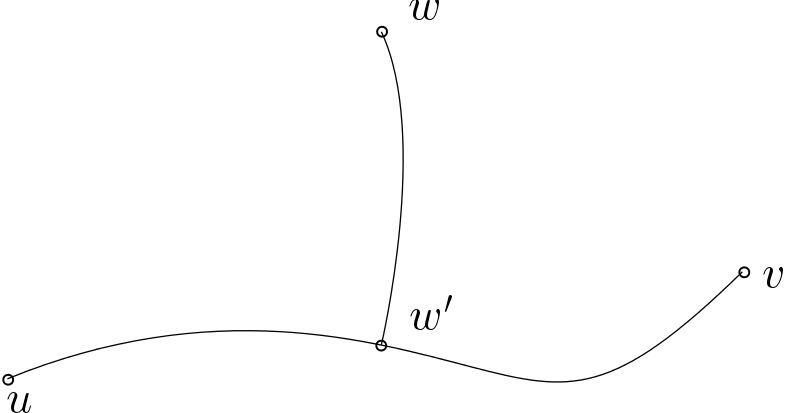}
\caption{Lemma \ref{colpreserv}: $\col_u(w)=\col_v(w)$ for all but finitely many $w \in R_0$.}
\end{figure}

Since $R$ is a tree, $P_{u,w} \cap P_{v,w} = P_{w',w}$ for some unique $w' \in P_{u,v}$.  Now if $P_{w',w}$ contains a consecutive pair, then  $\col_u(w)=\col_v(w)$ as desired. Otherwise $P_{w',w}$ must be a path with strictly decreasing labels starting with $\lab(w')$. But then there is only one such possible tree vertex $w$ for each (not necessarily tree) vertex $w'$ in the finite path $P_{u,v}$. 
\end{proof}

\begin{corollary}\label{colembpreserv}
Let $\phi$ be an embedding of $(R,r)$. Then $\col_r(v)=\col_{\phi(r)}(v) = \col_{\phi(r)}(\phi(v))$ for all but finitely many $v \in R_0$. \\
The only possible exceptions are vertices originating from a path starting from $P_{r,\phi(r)}$ with strictly decreasing labels.
\end{corollary}

\begin{proof}
By Lemma \ref{colpreserv}, $\col_r(v)=\col_{\phi(r)}(v)$, with only possible exceptions being vertices originating from a path starting from $P_{r,\phi(r)}$ with strictly decreasing labels. 
In addition, as remarked above,  $\col_r(v)=\col_{\phi(r)}(\phi(v))$ for any embedding $\phi$ and tree vertex $v$.
\end{proof}

The height of any (not necessarily tree) vertex $w$ with respect to another arbitrary vertex $v$ is the maximum label encountered in the path from $v$ to $w$. Again the construction will be done in stages determined by such maximum height. 

\begin{definition}\label{height}
For $v \in R$, define $\hth_v: R \rightarrow \N$, the height with respect to $v$,  by \[\hth_v(w)=max\{ \lab(w'): w' \in P_{v,w}\}.\] 
\end{definition}

Note that $\col_v(u) \leq \hth_v(u)$ for any tree vertices $v$ and $u$. Moreover $\col_v(u) = \hth_v(u)$ means that the label of the last consecutive pair in $P_{v,u}$ is the maximum label appearing among all vertices of $P_{v,u}$; this is actually the situation that will be used in the inductive construction, and for a given fixed vertex $v$ such a vertex $u$ will be called a {\em target vertex}.

It may also be worth noting that the corresponding Proposition  \ref{colpreserv} is clearly not true for the height value, meaning that there can be infinitely many (even tree) vertices $w$ such that $\hth_v(w) \neq \hth_u(w)$. However the following observation will be useful.

\begin{observation}\label{hthpreserv}
For any  vertices $u, v \in R$, $\hth_u(w)=\hth_v(w)$ for all $w \in R$ such that $min\{\hth_u(w), \hth_v(w)\}  \geq \hth_u(v)$.
\end{observation}
 
\begin{proof}
Under the given hypothesis we have 
\[ \hth_u(w) \leq max\{\hth_v(w),\hth_u(v)\} \leq  \hth_v(w). \]
By symmetry we have $\hth_v(w) \leq \hth_u(w)$ as well.
\end{proof}

Now for any tree vertex $v$, there is also an automorphism of $\SR=(R,v)$ fixing $v$ and interchanging its two neighbourhoods. In order to keep the number of siblings under control, we will want to prevent not only such automorphisms, but also embeddings mapping one neighbourhood into the other.   For this we will define a {\em sign function} (with respect to $v$) $\sign_v$ which takes values $+1$ on one neighbourhood of $v$, and $-1$ on the other neighbourhood, and eventually code these values as graph properties so they are preserved by embeddings. Moreover, within a neighbourhood, we will similarly want to control the neighbourhoods of each vertex and we will define a {\em spin function}  (with respect to $v$) for that purpose. The following will define these functions simultaneously, first by defining $\sign_r$, $\spin_r$, then $\sign_v$ and finally $\spin_v$ for all other tree vertices $v$.

We write $R^v_0= R_0\setminus\{v\}$ since these functions are not defined on the vertex they are based on. We will first define 

\begin{definition}\label{spinsign}
First arbitrarily assign  $\sign_r(u)=+1$ to every vertex $u$ in one neighbourhood of $r$, and $\sign_r(u)=-1$ to every vertex $u$ in the other neighbourhood of $r$.
\begin{itemize}
\item Let $v$ be a tree vertex and $\sign_v$ a sign function which assigns values $\pm 1$ to each neighbourhood of $v$. \\Then 
define $\spin_v: R^v_0 \rightarrow \pm 1$,  the spin with respect to $v$,  by:
\[ \spin_v(u)=\sign_v(u)(-1)^{cp+tv} \]
where 
\begin{itemize}
\item$cp=P^{cp}_{v,u}$ is the number of consecutive pairs in $P_{v,u}$
\item$tv=P^{tv}_{v,u}$ is the number of tree vertices in $P_{v,u}$. 
\end{itemize}

\item Let $v$ be a tree vertex, then define $\sign_v: R^v_0 \rightarrow \pm 1$ by:
\begin{itemize}
\item $\sign_v(u)= \spin_r(v)$ if $u$ and $r$ belong to the same neighbourhood of $v$, and 
\item $\sign_v(u)= -\spin_r(v)$  otherwise.
\end{itemize}
\end{itemize}
\end{definition}

\begin{figure}[ht]
\fbox{\includegraphics[width=\textwidth]{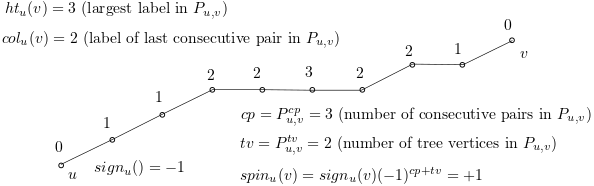}}
\caption{Definitions \ref{colour}, \ref{height} and \ref{spinsign}: Label, colour, height, sign and spin example.}
\end{figure}

Note that indeed  $\sign_v$ assigns the value $+1$ on one of its neighbourhood and $-1$ on the other. Moreover, the spin with respect to the root $r$ can be recovered from the sign at any other tree vertex, and hence also the spin at any other tree vertex as follows.

\begin{lemma}\label{spinpreserv}
Let $v \in R_0$. Then:
\begin{enumerate}
\item $\spin_v(r) = \sign_r(v)$ for any $v \in  R^r_0$. 
\item $\spin_v(w) = \spin_r(w)$ for all  $w \in (R^r_0 \cap R^v_0)  \setminus P_{r,v}$, and 
\item $\spin_v(w) = -\spin_r(w)$ for all  $w \in  P_{r,v} \setminus \{r,v\}$. 
\end{enumerate}
\end{lemma}

\begin{proof}
By Definition \ref{spinsign}, $\sign_v(r)= \spin_r(v)$ since $r$ and $r$ trivially belong to the same neighbourhood of $v$.
Thus, writing $cp=P^{cp}_{v,r}=P^{cp}_{r,v}$ and $tv=P^{tv}_{v,r}=P^{tv}_{r,v}$, we have 
 $\spin_v(r) = \sign_v(r) (-1)^{cp+tv} =
  \spin_r(v)(-1)^{cp+tv} = \left[ \sign_r(v) (-1)^{cp+tv} \right] (-1)^{cp+tv}
  = \sign_r(v) (-1)^{2cp+2tv} = \sign_r(v)$.
 
For (2), consider $w \in (R^r_0 \cap R^v_0)  \setminus P_{r,v}$. Then $P_{r,w} \cap P_{v,w} = P_{w',w}$ for some $w' \in P_{r,v}$. 

Assume first that $w'=v$, which means that $v$ is on the path from $r$ to $w$. Hence $P^{cp}_{r,w}= P^{cp}_{r,v}+P^{cp}_{v,w}$, and $P^{tv}_{r,w}=P^{tv}_{r,v}+P^{tv}_{v,w}-1$ because $v$ is counted twice as a tree vertex.  Moreover $\sign_r(w) = \sign_r(v)$ because $w$ and $v$ are in the same neighbourhood of $r$, and on the other hand $\sign_v(w) = -\spin_r(v)$ because $w$ and $r$ are in opposite neighbourhoods of $v$. Thus we get: 

\[ \begin{array}{ll}
\spin_r(w) & =  \sign_r(w)(-1)^{P^{cp}_{r,w}+P^{tv}_{r,w}} \\
 & = \sign_r(w)(-1)^{P^{cp}_{r,v}+P^{cp}_{v,w}+ P^{tv}_{r,v}+P^{tv}_{v,w}-1}  \\ 
 & = \left[ \sign_r(v)(-1)^{P^{cp}_{r,v} + P^{tv}_{r,v} -1} \right] (-1)^{P^{cp}_{v,w}+ P^{tv}_{v,w}} \\
 & = -\spin_r(v) (-1)^{P^{cp}_{v,w} +  P^{tv}_{v,w}} \\
 & = \sign_v(w) (-1)^{P^{cp}_{v,w} + P^{tv}_{v,w}} = \spin_v(w) 
 \end{array} \]

Similarly, $w'=r$ means that $r$ is on the path from $v$ to $w$. Hence $P^{cp}_{v,w}= P^{cp}_{v,r}+P^{cp}_{r,w}$, and $P^{tv}_{v,w}=P^{tv}_{v,r}+P^{tv}_{r,w}-1$ because here $r$ is counted twice as a tree vertex.  Moreover $\sign_v(w) = \spin_r(v)$ because $w$ and $r$ are in the same neighbourhood of $v$, and on the other hand $\sign_r(w) = -\spin_r(v)$ because $v$ and $w$ are in opposite neighbourhoods of $r$. Thus, also observing that  $P^{cp}_{r,v} = P^{cp}_{v,r}$ and $P^{tvp}_{r,v} = P^{tv}_{v,r}$,   we get: 
\[ \begin{array}{ll}
\spin_v(w) & =  \sign_v(w)(-1)^{P^{cp}_{v,w}+P^{tv}_{v,w}} \\
&  = \sign_v(w)(-1)^{P^{cp}_{v,r}+P^{cp}_{r,w}+ P^{tv}_{v,r}+P^{tv}_{r,w}-1}  \\
&  = \spin_r(v)(-1)^{P^{cp}_{v,r}+P^{cp}_{r,w}+ P^{tv}_{v,r}+P^{tv}_{r,w}-1}  \\
&  = \left[\sign_r(v)(-1)^{P^{cp}_{r,v}+ P^{tv}_{r,v}}\right] (-1)^{P^{cp}_{v,r}+P^{cp}_{r,w}+ P^{tv}_{v,r}+P^{tv}_{r,w}-1} \\
& =  \sign_r(v)(-1)^{2P^{cp}_{r,v}+ 2P^{tv}_{r,v}+P^{cp}_{r,w}+P^{tv}_{r,w}-1}  \\
& = - \sign_r(v)(-1)^{P^{cp}_{r,w}+P^{tv}_{r,w}}  \\
& =  \sign_r(w)(-1)^{P^{cp}_{r,w}+P^{tv}_{r,w}}   = \spin_r(w) 
\end{array} \]

\begin{figure}[ht]
\includegraphics[width=\textwidth]{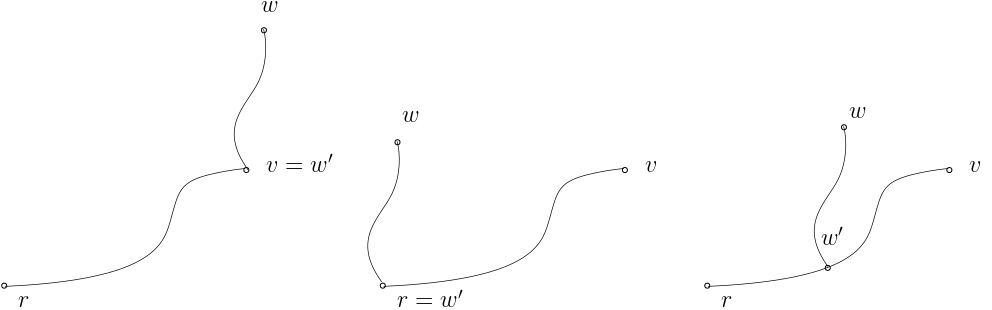}
\caption{Lemma \ref{spinpreserv}: The three cases: $w'=v$, $w'=r$, and $w' \in  P_{r,v} \setminus \{r,v\}$.}
\end{figure}

Finally assume that $w' \in  P_{r,v} \setminus \{r,v\}$. Note that in this case $w' \notin R_0$ since tree vertices have degree 2; this yields  $P^{cp}_{r,w}=P^{cp}_{r,w'}+P^{cp}_{w',w}$ and $P^{tv}_{r,w}=P^{tv}_{r,w'}+P^{tv}_{w',w}$, and similarly $P^{cp}_{v,w}=P^{cp}_{v,w'}+P^{cp}_{w',w}$ and $P^{tv}_{v,w}=P^{tv}_{v,w'}+P^{tv}_{w',w}$. 
If $P^{cp}_{r,w'}+P^{tv}_{r,w'}$ and $P^{cp}_{v,w'}+P^{tv}_{v,w'}$ have the same parity, this immediately implies $\spin_r(w)=\spin_v(w)$.
Otherwise,  $P^{cp}_{r,v}+P^{tv}_{r,v}$ is odd, and together with the fact that $r$ and $w$ are in the same neighbourhood of $v$, $v$ and $w$ are in the same neighbourhood of $r$,  all imply that $\sign_v(w)=\spin_r(v)=\sign_r(v)(-1)^{P^{cp}_{r,v}+P^{tv}_{r,v}}=-\sign_r(v)=-\sign_r(w)$; and 
since $P^{cp}_{r,w}+P^{tv}_{r,w}$ and $P^{cp}_{v,w}+P^{tv}_{v,w}$ have different parity we conclude that $\spin_r(w)=\spin_v(w)$.

The proof of (3) follows a similar analysis. Note here that $w \in R_0\setminus\{r,v\}$, and since $w$ and $r$ are in the same neighbourhood of $v$, we have:
\[ \begin{array}{ll}
\spin_v(w) & =  \sign_v(w)(-1)^{P^{cp}_{v,w}+P^{tv}_{v,w}} \\
&  = \spin_r(v)(-1)^{P^{cp}_{v,w}+P^{tv}_{v,w}} \\
&  = \left[ \sign_r(v)(-1)^{P^{cp}_{r,v}+P^{tv}_{r,v} }\right] (-1)^{P^{cp}_{v,w}+P^{tv}_{v,w}} \\
&  = \left[ \sign_r(v)(-1)^{P^{cp}_{r,w}+P^{cp}_{w,v}+ P^{tv}_{r,w} + P^{tv}_{w,v}-1 }\right] (-1)^{P^{cp}_{v,w}+P^{tv}_{v,w}} \\
&  = -\sign_r(v)(-1)^{2 P^{cp}_{v,w}+  2P^{tv}_{v,w}+  P^{cp}_{r,w}+  P^{tv}_{r,w}   } \\
&  = -\sign_r(w)(-1)^{P^{cp}_{r,w}+  P^{tv}_{r,w}   } \\
&  = -\spin_r(w)
\end{array} \]
This completes the proof of the lemma. 
\end{proof}

We can further correlate the spin between any two vertices. 

\begin{corollary}\label{spinpreserv3}
Let $u,v \in R_0$. Then 
\begin{enumerate}
\item $\spin_u(w)=\spin_v(w)$ for all  $w \in (R^u_0 \cap R^v_0)  \setminus P_{u,v}$, and
\item $\spin_u(w)=-\spin_v(w)$ for all  $w \in (R^u_0 \cap R^v_0)  \cap P_{u,v}$, with the only exception of $w \neq r$ and $P_{r,u}\cap P_{r,v}=P_{r,w}$, in which case $\spin_u(w)=\spin_v(w)$.
\end{enumerate}
\end{corollary}

\begin{proof}
The proof follows by carefully using Lemma \ref{spinpreserv}. Consider $w \in (R^u_0 \cap R^v_0)$.

If for a first case  $w \notin P_{u,r} \cup P_{v,r}$ (and hence $w \notin P_{u,v}$), then $\spin_u(w)=\spin_r(w)=\spin_v(w)$ by Lemma \ref{spinpreserv} (2). 

Next assume that $w \in P_{u,r} \cap P_{v,r}$. If $w \neq r$, then $w \in  P_{r,u} \setminus \{r,u\}$ and $w \in P_{r,v} \setminus \{r,v\}$, hence
$\spin_u(w)=-\spin_r(w)=\spin_v(w)$ by Lemma \ref{spinpreserv} (3). Note that this includes the case that $P_{r,u}\cap P_{r,v}=P_{r,w}$ with $w \in P_{u,v}$.
Now if $w=r$, then either $w \notin P_{u,v}$ and  thus $u$ and $v$ are in the same neighbourhood of $w=r$, so by Lemma \ref{spinpreserv} (1),
$\spin_u(w) = \spin_u(r)= \sign_r(u) = \sign_r(v) = \spin_v(r)= \spin_v(w)$; or else  $w=r \in P_{u,v}$, then $u$ and $v$ are in opposite neighbourhood of $w=r$, so again by Lemma \ref{spinpreserv} (1), 
$\spin_u(w) = \spin_u(r)= \sign_r(u) = -\sign_r(v) = -\spin_v(r)= -\spin_v(w)$. 

So finally assume without loss of generality that $w \in P_{u,r} \setminus P_{v,r}$, and hence $w \in P_{u,v}$.
Then $\spin_u(w) =-\spin_r(w)$ by Lemma \ref{spinpreserv} (3) since  $w \in P_{r,u} \setminus \{r,u\}$.
Also $\spin_v(w) =\spin_r(w)$ by Lemma \ref{spinpreserv} (2) since  $w \in (R^r_0 \cap R^v_0) \setminus P_{r,v}$.
Thus $\spin_u(w)=-\spin_v(w)$.
\end{proof}

Recall that labels have been encoded by graph propertie (with teh gadgets)s, and thus are preserved by any (rooted tree) embedding $\phi: (R,r) \rightarrow (R,v)$. Hence the height and colour functions are also preserved. If we further ask to preserve the spin (and hence sign) function, then such a \semb 
  as we denote it is unique (except for possibly interchanging the leaves of gadgets, which is immaterial for our purpose); we will see later how the construction will yield a tree actually preserving the spin through graph properties. 

\subsection{Ray vertices}\label{rayvertices}
The second type of structures that will be used in the construction are double rays (two-way infinite rays) with two kinds of vertex types.
Eventually, each tree vertex in all copies of the graph $R$ will be amalgamated with a vertex on a double ray. 

Let $D$ be a double ray, which we will normally identify as $D=\{v_i: i \in \Z \}$ with edges $(v_i,v_{i+1})$ for $i \in \Z$. We will equip $D$ with {\em type assignments} of the form $\type : V(D) \rightarrow \{0,1\}$, and we write $\SD_{\type}=(D,\type)$  for the resulting structure. A {\type}-embedding  (we could also simply use \semb ) $\phi:\SD_{\type} \rightarrow \SD_{\type'}$ is a graph embedding of $P$ such that $\type(v_i) \leq \type'(\phi(v_i))$ for each $i \in \Z$. 
What will be important is that an  embedding of $\SD_{\type}$ can send type $0$ vertices into type $1$ vertices, but not the other way around.  
Again we encode these structures as graph properties, and we do so by identifying every vertex $v \in D$ with the root of a copy of $PK(2,2)$ if $\type(v)=0$, and with the root of $PK(2,3)$ if $\type(v)=1$. We remark that eventually every tree vertex $v$ will be a vertex in a double ray equipped with a type assignment, and so will have two finite gadgets attached to it:  one that encodes its label of zero as a tree vertex, that is a copy of $PK(2,2)$,  and the other identified with the root of a copy of either $PK(2,2)$ or $PK(2,3)$  matching its type assignment as a ray vertex. Graph embeddings will necessarily send label attachments and type attachments to attachments of the same kind.   Since $PK(2,2)$ embeds in $PK(2,3)$ but not vice versa the type property is preserved as well as the label.

We will loosely call $\SD_{\type}$ a double ray even though it comes equipped with vertices of type $0$ or $1$, and we will continue to use the symbol $P$ to denote a regular double ray (without any type assignment).
Now consider the special type assignment $\type_0$ such that:
\[ \type_0(v_i) = \left\{ \begin{array}{l}0 \mbox{ for }i \leq 0 \\ 1 \mbox{ for  } i > 0. \end{array} \right. \]
and let $\SD_0=\SD_{\type_0}$ be the resulting double ray. Note that $v_0$ is the first (only in this case) vertex of type $0$ followed by a type $1$ vertex, and we call it the center $z=z_0$ of $\SD_0$. Observe  that all siblings of $\SD_0$ are of the form $\SD_{\type}$ for some type assignment $\type$ consisting of a  finite modification of the  above type  $\type_0$. Hence there are exactly countably many  (up to isomorphy) pairwise non-isomorphic siblings of $\SD_0$ not isomorphic to $\SD_0$, and all will have a vertex of type  $1$ followed by a type $0$ vertex. We select  $\langle \SD_s: 0<s < \sss \rangle$ with pairwise non-isomorphic type  assignments $\langle \type_s: 0<s < \sss \rangle$. For example we can select $\type_s$ as follows: 
\[ \type_s(v_j) = \left\{ \begin{array}{l}0 \mbox{ for }j < 0 \mbox { or } 1 \leq j \leq s \\ 1 \mbox{ for  } j = 0 \mbox{ or } j \geq s+1. \end{array} \right. \]
We again let $z_s=v_0$ be the centre of $\SD_s$.

\begin{figure}[ht]
\includegraphics[width=.9\textwidth]{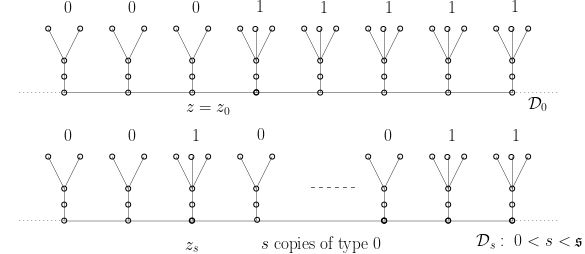}\label{doublerays}
\caption{Double rays $\SD_{\type_s}: s < \sss$ with their (graph version) gadget 0-1 type assignments}
\end{figure}

This indexing is designed specifically  so no embedding from $\SD_s$ to $\SD_{s'}$ can send $z_s$ to $z_{s'}$ for any $s \neq s'$. These will also form  the centres of the resulting trees $\ST_s$ and will be used to ensure embeddings preserve the sign and spin functions as promised earlier. 
  
We will call {\em ray vertices} those vertices $v_i$'s on the double rays. Because ray vertices of type $1$ cannot embed in ray vertices of type $0$, then these double rays are equipped with a natural direction dictated by embeddings, reflecting the positive direction of the indexing along  $\Z$. 
  
Finally if $D$ is a double ray in a tree $T$ and $v \in D$, we define for later convenience  $T^D_v$ the connected component of $T$ containing $v$ without its two neighbours on the double ray $D$.

Before we move on to other required properties of these trees, we note that the case of partial orders will require that the double ray type assignments above be done on even indexed vertices only, and that the odd indexed vertices all be equipped with a gadget that does not embed in any the type assignments gadgets, and vice-versa: using $PK(4,2)$ on odd indexed vertices for example will do (which is why we left that gadget available). 

\begin{figure}[ht]
\includegraphics[width=.9\textwidth]{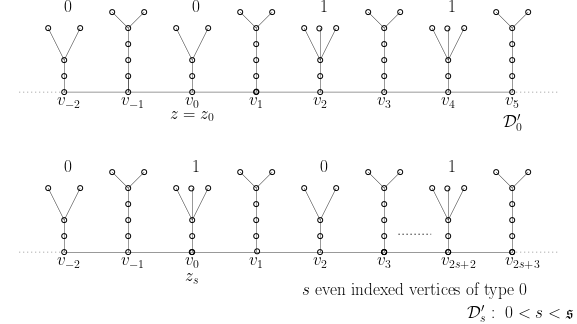}\label{doubleraysposets}
\caption{Double rays $\SD'_{\type_s}: s < \sss$  for Posets with gadget 0-1 type assignments on even index vertices.}
\end{figure}

We urge the reader to take note that the construction below can be done using either kind of type assignments on the double rays with similar results; this will be used for the case of partial orders. 

\subsection{Global colouring, spin and height}
Each of the trees $\langle \ST_s : s <\sss \rangle$ will be built in a similar manner, and we will do so by first constructing a common ``spine'' $\Sp$. 
This will be done by first assembling disjoint copies of $(R,r)$ along a double ray $\SD$, identifying each ray vertex of $\SD$ with the root of a disjoint copy of $(R,r)$. 
Then each existing tree vertex (other than the root, r,  in the copies of $(R,r)$) will be identified with a vertex of a disjoint copy of a double ray which we will ensure is the centre of that double ray, and each new ray vertex will be identified with the root of a disjoint copy of $(R,r)$, etc. Then later all we will need to do is for each $ s <\sss$  add judiciously chosen type assignments on each of those double rays to obtain $\ST_s$. 

Every copy of $(R,r)$ inherits its own local labelling, colour, height, sign and spin functions, and we will wish to extend these notions globally to $\Sp$ (and eventually $\ST_s$); we will introduce a centre of $\Sp$ when the need arises. The labelling of a vertex in $\Sp$ will simply be its labelling when considered within its own copy of $(R,r)$, but we wish to extend the other notions globally to $\Sp$ so they apply across copies of $(R,r)$. 

We will proceed to build $\Sp$ in stages $\Sp(k)$, but first we define a global height  $\whth$ in such trees formed by assembling disjoint copies of $(R,r)$. Recall that each such copy comes equipped with its corresponding label function, and hence the global height is simply the maximum label encountered in a path. This will be used in particular to determine the stages of the construction. 
 
\begin{definition}\label{globalheight}
Let $\SSS$ be a tree formed by assembling disjoint copies of $(R,r)$ along double rays. Then for $v,w \in \SSS$, define 
\begin{center}
\begin{tabular}{ll}
$\whth_v(w) =$ & $max \{\lab^{R'}(w'): w' \in P_{v,w}$, and \\
& $w'$ is in a copy $(R',r')$ of  $(R,r) \}.$ \\
\end{tabular}
\end{center}
\end{definition} 

Note that $\whth_v(w) = \hth_v(w)$ in case $v$ and $w$ belong to the same copy of $(R,r)$. We are now ready to  define the spine $\Sp$.

\begin{definition}\label{spine}
\begin{itemize}
\item We \underline{activate} a ray vertex by identifying that vertex with the root $r'$ of a disjoint copy $(R',r')$ of $(R,r)$.
\item We \underline{amalgamate} a tree vertex by identifying that vertex with one of the vertices of a disjoint copy of a double ray.
\item Define $\Sp(0)$ by activating every ray vertex of a double ray.  \\
 Given $\Sp(k-1)$, let $\Sp_0(k)= \Sp(k-1)$, and $\Sp_{\ell+1}(k)$ be obtained from  $\Sp_\ell(k)$ by amalgamating every non-amalgamated tree vertex of global height at most $k$, followed by activating every new ray vertex. \\
Define  $\Sp(k) = \bigcup_\ell \Sp_\ell(k)$.
\item
Finally let $\Sp =\bigcup_k \Sp(k)$.
\end{itemize}
\end{definition} 

Observe that every ray vertex of  $\Sp$  is {\em activated} (with the root of a disjoint copy of $(R,r)$), thus we can think of every vertex of  $\Sp$ as being in a copy of $(R,r)$. Moreover every tree vertex is {\em amalgamated} with a vertex of a double ray, and this is through those double rays that one navigates from one copy of $(R,r)$ to another. 

Now as before an embedding $\phi$ of $\Sp$ must preserve labels, that is $\lab^{R'}(v)=\lab^{R''}(\phi(v))$ where $v$ belongs to a copy $(R',r')$ and $\phi(v)$ to a copy $(R'',r'')$; this is simply because the finite trees attached to tree vertices do not embed into each other for different labels. 
But this means that $\phi$ actually preserves copies since moving from one copy to another requires a path through at least two consecutive ray vertices, which are activated to tree vertices of label 0. We take note of this in the following  important observation. 

\begin{observation}\label{spinepreservecopies}
Any embedding of $\Sp(k)$ or $\Sp$ is surjective, preserves labels, ray and tree vertices, and copies of $(R,r)$. 
\end{observation}

We next extend the colour function globally to $\Sp$, and this is a bit more delicate. First note that new copies of $(R,r)$ are created by 
 identifying a ray vertex with the root $r'$ of a disjoint copy $(R',r')$ of $(R,r)$. But once that copy is created, a path originating from outside it could enter that copy through a different tree vertex $v' \neq r'$ that was later activated. 

\begin{definition}\label{globalcolour}
\begin{enumerate}
\item Let $\wR_0 = \{v \in \Sp: v $ belongs to a copy $(R',r')$ of $(R,r)$ and $\lab^{R'}(v)=0\}$. 
\item For $v \in \wR_0$, define $\wcol_v: \wR_0  \rightarrow \N$ by $\wcol_v(w) = \col^{R'}_{v'}(w)$ where $w$ belongs to a copy $(R',r')$ of $(R,r)$ and  $v'$ is the first vertex of $P_{v,w} \cap R'$.
\end{enumerate}
\end{definition}

\begin{figure}[ht]
\includegraphics[width=.5\textwidth]{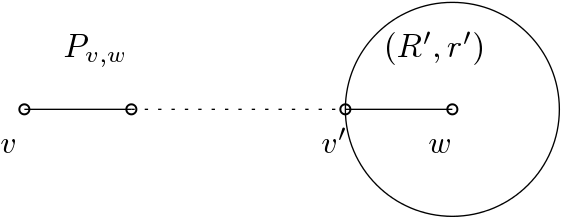}
\caption{Definition \ref{globalcolour}: The global colour function:  $\wcol_v(w) = \col^{R'}_{v'}(w)$. }
\end{figure}

Thus $v'=v$ and $\wcol_v(w) = \col^{R'}_v(w)$ in case $v$ and $w$ belong to the same copy $(R',r')$ of $(R,r)$. We now show that Lemma \ref{colpreserv} generalizes globally as follows. 

\begin{lemma}\label{globcolpreserv}
For any  two tree vertices  $u,v \in \wR_0$,  $\wcol_u(w)=\wcol_v(w)$ for all but finitely many $w \in \wR_0$.  \\
The only possible exceptions are vertices originating from a path starting from $P_{u,v}$ with strictly decreasing labels. 
\end{lemma}

\begin{proof}
Let $P_{u,w} \cap P_{v,w} = P_{w',w}$ for some $w' \in P_{u,v}$. If $P_{w',w}$ contains a consecutive pair in $R'$, then $\wcol_u(w)=\wcol_v(w)$.
Otherwise  $P_{w',w}$ is a path (in some copies $(R',r')$) with strictly decreasing labels, and observe there is  only one such possible vertex $w$ for each $w' \in P_{u,v}$. 
\end{proof}

It will later be useful to understand the global colouring through embeddings, we do so in two parts.

\begin{figure}[ht]
\includegraphics[width=.7\textwidth]{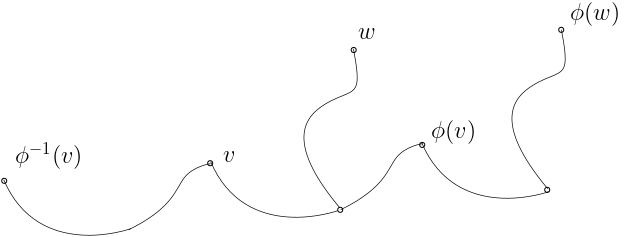}
\caption{Corollaries  \ref{embcolpreserv1} and \ref{embcolpres2}: $\wcol_v(w)$, $\wcol_{\phi(v)}(w)$, $\wcol_{\phi(v)}(\phi(w))$, and $\wcol_v(\phi(w))$. }
\end{figure}

\begin{corollary}\label{embcolpreserv1}
Let $v \in \wR_0$ be a tree vertex, and $\phi$ an embedding of $\Sp$. Then $\wcol_v(w)=\wcol_{\phi(v)}(w)=\wcol_{\phi(v)}(\phi(w))$ for all but finitely many $w \in \wR_0$. \\
The only possible exceptions are vertices originating from a path starting from $P_{v,\phi(v)}$ with strictly decreasing labels.
\end{corollary}

\begin{proof}
$\wcol_v(w)=\wcol_{\phi(v)}(\phi(w))$ since $\phi$ preserves labels, and $\wcol_v(w)=\wcol_{\phi(v)}(w)$ by Lermma \ref{globcolpreserv} assuming  that $w$  does not originate from a path starting from $P_{v,\phi(v)}$ with strictly decreasing labels.
\end{proof}

\begin{corollary}\label{embcolpres2}
Let $v \in \wR_0$ be a tree vertex, and $\phi$ an embedding of $\Sp$. Then $\wcol_v(w)=\wcol_v(\phi(w))$ for all but finitely many $w \in \wR_0$. \\
The only possible exceptions are vertices originating from a path starting from $P_{\phi^{-1}(v),v}$ with strictly decreasing labels.
\end{corollary}

\begin{proof}
Let $v \in \wR_0$ be a tree vertex, $\phi$ an embedding of $\Sp$, and $w$ not originating from a path starting from $P_{\phi^{-1}(v),v}$ with strictly decreasing labels. Recall by Observation \ref{spinepreservecopies} that $\phi$ is surjective. 
Hence,  by Corollary \ref{embcolpreserv1},  $\wcol_{\phi^{-1}(v)}(w)=\wcol_{\phi \circ\phi^{-1}(v)}(w)=\wcol_{\phi \circ\phi^{-1}(v)}(\phi(w))$. 
Thus we conclude that  $\wcol_v(w)=\wcol_v(\phi(w))$.
\end{proof}

Similarly we define a global sign and spin functions. First recall that $\sign_v$ and  $\spin_v $ are undefined at $v$, and we take this into account in the global setting for all copies of $(R,r)$; thus here the global spin at $v$ will not be defined on the first vertex from a path from $v$ to a copy of $(R,r)$.

\begin{definition}\label{globspin}
Let $v \in \wR_0$, and define $\wR_0^v =\{w \in \wR_0:  w$ belongs to a copy $(R',r')$ of $(R,r)$ and  $w$ is not the first vertex of $P_{v,w} \cap R' \}$. 
\begin{enumerate}
\item Define $\wsign_v: \wR_0^v \rightarrow \pm 1$ by $\wsign_v(w) = \sign_{v'}(w)$ where $w$ belongs to a copy $(R',r')$ of $(R,r)$ and  $v'$ is the first vertex of $P_{v,w} \cap R'$.
\item Simiarly define $\wspin_v: \wR_0^v \rightarrow \pm 1$ by $\wspin_v(w) = \spin_{v'}(w)$ where $w$ belongs to a copy $(R',r')$ of $(R,r)$ and  $v'$ is the first vertex of $P_{v,w} \cap R'$.
\end{enumerate}
\end{definition}

Thus in the case that $v$ and $w$ belong to the same copy of $(R,r)$, then $v'=v$, $\wsign_v(w) = \sign_v(w)$  and $\wspin_v(w) = \spin_v(w)$.   

\begin{figure}[ht]
\includegraphics[width=.5\textwidth]{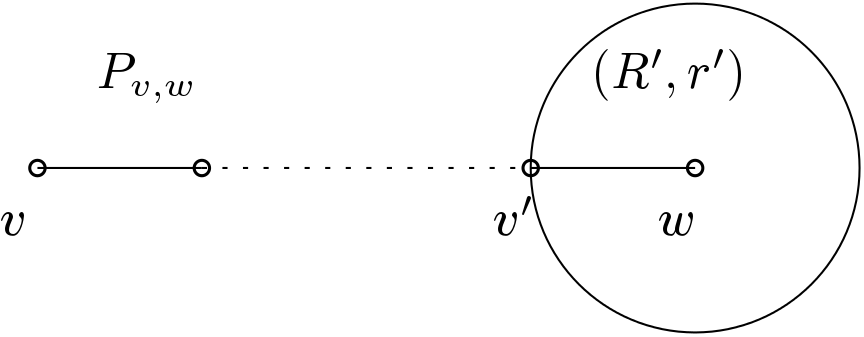}
\caption{Definition \ref{globspin}: The global sign and spin functions: $\wsign_v(w) = \sign_{v'}(w)$ and $\wspin_v(w) = \spin_{v'}(w)$.}
\end{figure}

\begin{lemma}\label{globalspinpreserving}
For any  tree vertices $u,v \in  \wR_0$,  $\wspin_u(w)=\wspin_v(w)$ for all but finitely many $w \in \wR_0^u \cap \wR_0^v$, in fact for all $w \notin P_{u,v}$. 
\end{lemma}

\begin{proof}
Suppose $w \notin P_{u,v}$, $(R',r')$ a copy of $(R,r)$ containing $w$, $u'$ is the first vertex of $P_{u,w} \cap R'$ and similarly $v'$ is the first vertex of $P_{v,w} \cap R'$.

\begin{figure}[ht]
\includegraphics[width=.5\textwidth]{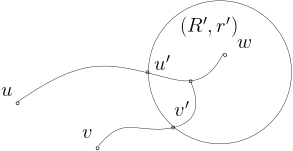}
\caption{Lemma \ref{globalspinpreserving}: $\wspin_u(w) = \spin_{u'}(w)= \spin_{v'}(w) = \wspin_v(w)$.}
\end{figure}

But now $ w \notin P_{u',v'}$ since $w \notin P_{u,v}$ and $w \neq u',v'$. So $\wspin_u(w)=\spin_{u'}(w)=\spin_{v'}(w)=\wspin_v(w)$ by
Corollary \ref{spinpreserv3} and Definition \ref{globspin}.
\end{proof}

\begin{observation}\label{embeddingheightpreserving}
Let $\phi$ an embedding of $\Sp$, then $\whth_v(u)=\whth_{\phi(v)}(\phi(u))$ for all $u,v \in \SD$. 
\end{observation}

As observed before, any graph embedding of $\Sp$ is surjective and hence an automorphism, and thus $\sib(\Sp)=1$. The trees  $\ST_s$ will be obtained from $\Sp$ by judiciously setting type assignments to ray vertices of $\Sp$, and as a result  embeddings of $\ST_s$ will preserve the global sign and global spin functions. This will be the main tool in showing that $\sib(\ST)=\sss$, and we are now ready to undertake the construction of the trees $\ST_s$.

\subsection{The trees $\langle \ST_s(k): s < \sss \rangle$}

For a non-zero $\sss \in \N$, the final trees $\langle \ST_s: s < \sss \rangle$  we are seeking to produce will be constructed in similar manners to $\Sp$, as a countable union of trees $\ST_s(k)$ for $k \in \N$. The trees $ \ST_s(k)$  will consist of the spine $\Sp(k)$ together with type assignments to double rays amalgamated to tree vertices of global height at most $k$; interestingly, the only difference among the various $\ST_s(k)$'s is the original type assignment to the first double ray $\SD_s$. We will often simply write $\ST$ for $\ST_0$ and similarly  $\ST(k)$ for  $\ST_0(k)$.

Recall that every vertex of $\Sp$ belongs to a copy  $(R',r')$  of $(R,r)$, each inheriting its own corresponding collection of labels, colours, height, sign  and spin functions   from $\SR$. The notions of colours and height are graph properties, and we have seen in Subsection \ref{rootedtree} how we have encoded the labels as graph properties (through connecting a vertex of label $\ell$ with the root of a gadget being a path of length $2\ell+6$ whose end point is identified with $u \in K_{1,2}=(u,V)$; these finite graphs do not embed into each other as rooted trees unless they are equal). We have also seen in Subsection \ref{rayvertices} how we encode the type assignments as graph properties (through identifying a ray vertex of type 0 with the root of the gadget $PK(2,2)$,  and a ray vertex of type 1 with a leaf of $PK(2,3)$; the type 0 gadgets  embed into type 1 gadgets as a rooted tree, but not the other way around).
It is important to remind ourselves that the label gadgets and the type gadgets must be preserved by embeddings with type gadgets of type 0 possibly embedded in ones of type 1, resulting in one vertex being omitted from the image of the embedding.  It will remain to show how the sign and spin functions can be encoded through graph properties. But first we will build the trees and show how to handle their siblings. 

To assist with the construction, the trees will have a distinguished vertex which we call the {\em centre} of the tree. We define $\ST_s(0)$ by identifying every ray vertex of $\SD_s$ with the root $r'$ of a disjoint copy $(R',r')$ of $(R,r)$ (see Figure \ref{baseTs0}).

\begin{figure}[ht]
\includegraphics[width=.9\textwidth]{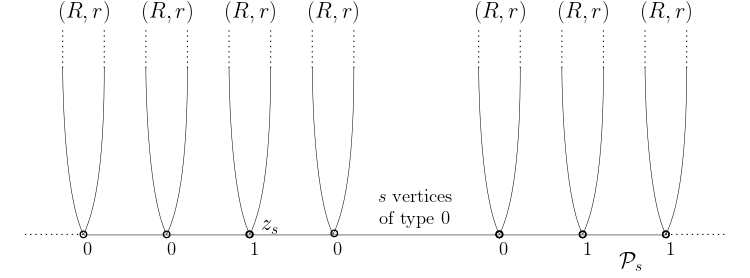}\label{baseTs0}
\caption{$\ST_s(0)$}
\end{figure}

Observe first that the spine of $\ST_s(0)$, that is the tree obtained from $\ST_s(0)$ by deleting the type assignments on ray vertices (on $\SD_s$), is $\Sp(0)$ as previously defined. Moreover, the spine of any sibling of $\ST_s(0)$ is the same $\Sp(0)$ (up to isomorphy), and this is because any self-embeddings of $\ST_s(0)$ will map  $\Sp(0)$  onto itself. Observe further that $\langle \ST_s(0): s<\sss \rangle$ are pairwise non-isomorphic siblings, and have (up to isomorphism) countably many siblings, each one represented  by a type assignment (consisting of only finitely many modifications of $\type_0$) on $\SD_s$, the ray vertices of the spine  $\ST_s(0)$. 
We further define the vertex $z_s \in \SD_s$ as the {\em centre} of the tree $\ST_s(0)$. These will remain the centres of all trees  $\langle \ST_s(k):s < \sss \rangle$ which we are about to define. 

The inductive construction will use a more delicate operation to {\em freely amalgamate} two trees over a common subtree. This situation generally occurs for example when two mathematical objects of similar types share a common substructure and the remaining structure does not interfere with each other, hence one may want to combine them together. We define here the specific case we will need.  

\begin{definition}\label{amalgamate}
Let  $a_0$ be a vertex of a tree $A_0$,  $a_1$ a vertex of a  tree $A_1$. Assume further that $a$ is a vertex of a tree $A$ and we have rooted tree embeddings  $\phi_i:(A,a) \rightarrow (A_i,a_i)$. 

Then we say that a tree $B$ is a \underline{(free) amalgamation} of $(A_0,a_0)$ and $(A_1,a_1)$ over $(A,a)$, or  $(A_0,a_0)$ and $(A_1,a_1)$ \underline{(freely) amalgamate} (over $(A,a)$),  if there are  embeddings  $\phi'_i:A_i \rightarrow B$  such that:
\begin{enumerate}
 \item $B = \phi_0'(A_0)\cup \phi_1'(A_1)$, 
 \item $ \phi_0' \circ \phi_0 = \phi_1' \circ \phi_1$,  and 
 \item $\phi_0'(A_0)\cap \phi_1'(A_1)=\phi_0' \circ \phi_0(A) = \phi_1' \circ \phi_1(A)$.
\end{enumerate}
\end{definition}

\begin{figure}[ht]
\includegraphics[width=.4\textwidth]{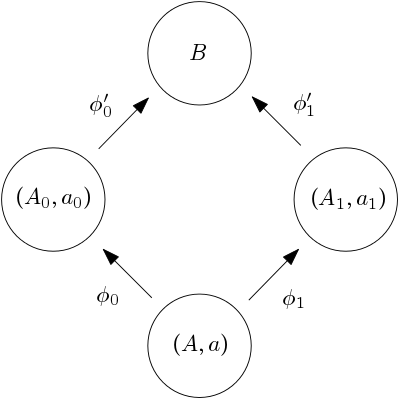}\label{amalgamate}
\caption{$B$ is an {\em amalgamation} of $(A_0,a_0)$ and $(A_1,a_1)$ over $(A,a)$.}
\end{figure}

Note that since $B$ is a tree and $A$ is non-empty, then no edges are added to $\phi_0'(A_0)\cup \phi_1'(A_1)$; the two pieces are simply joined together identifying their common copy of $A$. A simple   example is to observe  that identifying a ray vertex $v$ with the root of a copy of $(R,r)$ as we have done can be expressed as amalgamating a double ray $P$ containing $v$ and $(R,r)$ over $v$ (mapping $v$ to $r$). But we will more generally amalgamate larger trees in the construction, hence the need for the more general concept above. 

The construction will insure that the $\ST_s(k)$ are non-isomophic siblings, and moreover that $\sib(\ST(k))=\aleph_0$ for all $k \in \N$ (and hence 
 $\sib(\ST_s(k))=\aleph_0$  for all $s < \sss$). 
Once $\ST(k)$ has been constructed, we list  (representatives of) the pairwise non-isomorphic siblings of $\ST(k)$ not isomorphic to any $\ST_s(k))$  as $\{S_{k,\ell}: \ell \in \N \}$. Note that by considering those siblings as substructures of $\ST(k)$, they will come equipped with all the structure from $\ST(k)$; in particular we will select in $S_{k,\ell}$ a double ray non-isomorpic to $\SD_0$, and fix the centre $c_{k,\ell}$ of $S_{k,\ell}$ as the first vertex on that double ray having type $1$ followed by a vertex of type $0$ (as is the case for each $z_s$ for $s>0$). This is again for the same reason as before to ensure that no embedding of this double ray  into $\SD_0$  can send centres to centres, and vice versa; this is what will be used to show that the spin is preserved by embeddings. We will later justify the existence of such a double ray in all non-isomorphic siblings of $\ST(k)$.

As mentioned, all siblings of the (eventual) tree $\ST=\ST_0$ will be isomorphic to some  $\ST_s$, and to do so we will amalgamate approximations of those siblings within each $\ST_s(k)$ along the way. It turns out that this will suffice because siblings will differ from $\ST$ by only finitely many type assignments and hence will be captured at some stage. 

\subsubsection{Inductive construction}

The following notions will help better describe the construction.

\begin{definition}\label{targetvert}
\begin{itemize}
\item A  tree vertex $v \in \ST_s(k)$ is called a \underline{target} vertex (of global height $\ell$) if $\whth_{z_s}(v)=\wcol_{z_s}(v)=\ell$ for some $\ell\geq 0$. 
\item A crater (or $\ell$-crater) centered at a target vertrex $v$ of global height $\ell$, written $\SC(v)$,  consists of all vertices of global height less than $\ell$ from $v$. That is $\SC(v)=\{u \in \ST_s(k): \whth_v(u) < \ell \}$.
\item We say that a tree vertex $v \in \ST_s(k)$ has been \underline{amalgamated} if it was part of an amalgamation.  
\end{itemize}
\end{definition}

Thus a target vertex $v$ of global height $\ell$ is the end vertex in a copy of $(R,r)$ of a path $P_{z_s,v}$ originating at the center $z_s$ having its last consecutive pair of highest labels $\ell$ among the path labels and with decreasing labels from that consecutive pair to $v$. Then all vertices in its crater $\SC(v)$ have global height $\ell$ with respect to $z_s$. 

When a tree vertex is amalgamated, it will be identified with a ray vertex from a double ray and provided with a type assignment. The terminology to {\em amalgamate} a tree vertex is thus consistent with that of Definition \ref{spine}.

We consider the root $r$ of $(T,r)$ as a target vertex of (global) height 0, and for each $s< \sss$ it has been amalgamated to the centre of the double ray $\SD_s$, with each ray vertex amalgamated to the root of a copy of $(R,r)$, producing $\ST_s(0)$. There are no other target vertices of global height 0 in $\ST_s(0)$ and thus we can state that all target vertices of $\ST_s(0)$  have been amalgamated up to global height $0$ with respect to their centres $z_s$. 

We now define how to extend each tree $\ST_s(k-1)$ from a stage $k-1$ to $\ST_s(k)$ at the next level $k$; this is done the same way for all $s < \sss$ and as a result all trees rooted at ray vertices on all $\SD_s$ will be identical. Assume that the trees $\langle \ST_s(k-1): s < \sss  \rangle$ have been constructed for some $k \geq 1$, that  all tree vertices are amalgamated up to global height $k-1$ with respect to their centres $z_s$, and that the spine of any sibling of $\ST_s(k-1)$ is (up to isomorphy) $\Sp(k-1)$.

Write $k=2^i(2j+1)$, and consider $S=S_{i,j}$, a sibling of $\ST(i)$, with centre $c=c_{i,j}$ lying on a double ray $\SD$ (so that no embedding into $\SD_0$ can send $c$ to $z=z_0$ and vice versa). Considering $S$ as a substructure of  $\ST(i)$, and hence $\ST(k-1)$ since $i<k$, we can extend $S$ within $\ST(k-1)$ following the inductive construction, and assume that all tree vertices in $S$ are amalgamated up to global height $k-1$ with respect to its center $c$. Now fix $s < \sss$ and consider a {\em target} vertex $v \in\ST_s(k-1)$ such that $\whth_{z_s}(v)=\wcol_{z_s}(v)=k$, and belonging to a copy $(R',r')$ of $(R,r)$. Thus $v$ is not yet amalgamated. Moreover, since $\wcol_{z_s}(v)=k>0$, the last consecutive pair on $P_{z_s,v}$  has labels $k$ and therefore (from the perspective of $v$) all tree vertices having global height less than or equal to $k-1$ with respect to $v$, namely all vertices in its crater $\SC(v)$,  are of global height greater than or equal to $k$ with respect to ${z_s}$, and are thus also not yet amalgamated. 
On the other hand, since $S$ is a substructure of  $\ST(k-1)$, all vertices in $S$ having global height larger than or equal to $k$ with respect to $c$ are not yet amalgamated. Thus both $S$ and $\ST_s(k-1)$ can safely be amalgamated over $(R',r')$ by identifying $v$ and $c$.   The result of the amalgamation will be that all ray vertices in the crater $\SC(v)$ will receive a type assignment from $S$. The same is true of $(\ST(k-1), z_0)$ replacing $(S,c)$. 

\noindent Thus we can proceed as follows, and either amalgamate:
\begin{itemize}
\item  $(\ST_s(k-1),v)$ with $(S,c)$ over $(R',r')$,  if $\wspin_{z_s}(v)=+1$.
\item  $(\ST_s(k-1),v)$ with $(\ST(k-1),z_0)$ over $(R',r')$,  if $\wspin_{z_s}(v)=-1$.
\end{itemize}
Thus if one considers the existing copy $(R',r')$ from $\ST_s(k-1)$ as being rooted at $v$ for a moment, it becomes amalgamated with the copy of  $(R,r)$ from $(S,c)$ (or $(\ST(k-1),z_0)$) rooted at $c$ (or $z_0$ respectively), in particular matching the corresponding  sign functions at $u$ and $c$ from both copies.  Then $v$ becomes amalgamated, and this creates what we call $\ST_{s,v}(k-1)$. Note that we indeed amalgamate $(\ST(k-1),z_0)$ with $(\ST_s(k-1),v)$ for any $s < \sss$, thus the process is the same no matter which  $s$.

\begin{figure}[ht]
\includegraphics[width=.5\textwidth]{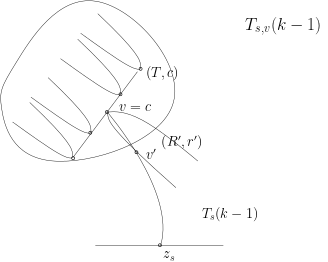}\label{amalgamatet}
\caption{$\ST_{s,v}(k-1)$: amalgamate $(\ST_s(k-1),v)$ with $(T,c)$ over $(R,r)$} 
\end{figure}

Now we repeat the same construction for all target vertices of $\ST_s(k-1)$ of global height $k$ with respect to $z_s$. Observe that any two such distinct target  vertices are separated by a consecutive pair of labels $k$, and therefore the corresponding craters do not intersect and their amalgamation as described above does not interfere with each other. But these amalgamations introduce  new target vertices of global height $k$ with respect to $z_s$ in the resulting amalgamated trees, so we repeat until all such vertices have been amalgamated. First define $\ST_s^0(k-1)=\ST_s(k-1)$, and:
\[ \ST_s^{\ell+1}(k-1)= \bigcup \{ \ST^\ell_{s,v}(k-1):  v \in \ST_s^\ell(k-1) \mbox{ and } \whth_{z_s}(v)=\wcol_{z_s}(v)=k\}.\]
Finally define
\[ \ST_s(k) = \bigcup_{\ell \in \N} \ST_s^\ell(k-1).\]

This completes the inductive construction.

\medskip

Obverse that the construction ensures that the spine of $\ST_s(k)$ is $\Sp(k)$. 
At stage $k$ the trees $\ST_s(k)$  contain the following types of vertices:
\begin{enumerate}
\item Ray vertices of type assignment $0$ or $1$, and all ray vertices have global height at most $k$ with respect to $z_s$; \\
All ray vertices are amalgamated (activated) with the root of a copy of $(R,r)$;
\item Target tree vertices amalgamated to centres of extended trees of the form $S=S_{i,j}$ where $2^i(2j+1)=\ell \leq k$, or of the form $\ST(\ell-1)$ for some $\ell \leq k$; these are tree vertices $v$ such that $\whth_{z_s}(v)=\wcol_{z_s}(v)=\ell$.
\item Amalgamated tree vertices occurring within copies of trees of the same form $S$ or $\ST(\ell-1)$, and themselves amalgamated to a target vertex at its centre;  these are tree vertices $v$ such that $\wcol_{z_s}(v) < \whth_{z_s}(v)=\ell \leq k$.
\item Not yet amalgamated target vertices, these are tree vertices $v$ such that $\whth_{z_s}(v)=\wcol_{z_s}(v)>  k$.
\item Not yet amalgamated tree vertices, these are tree vertices $v$ such that $\whth_{z_s}(v)>  k$.
\end{enumerate}

Note also that $\ST_s(k)$ is the disjoint union of the craters centered at target vertices, and this will be useful in discussing and creating embeddings. 
In the next section we will justify the construction, in particular showing that the number of siblings of each $\ST_s(k)$ is countable.

\subsubsection{Justification of the inductive construction}
At this point the trees $\ST_s(k)$ are equipped with (finite trees coding) labels on all vertices in copies of $(R,r)$, (finite trees coding) type assignments  on ray vertices,  and signs and spins functions on tree vertices. Due to the finite trees we have seen that the first two notions are graph properties and are thus preserved by (graph) embeddings. We now show that (graph) embeddings also preserve amalgamated and non-amalgamated vertices, and it will remain to show that the (global) sign and spin functions can also be recovered from the graph structure. 

\begin{lemma}\label{embedkcopies} 
Let $\phi$ be a (graph)  self-embedding of $\ST(k)$ (with centre $z=z_0$).

\begin{enumerate}
\item  Then $\phi$ preserves amalgamated tree vertices; that is maps amalgamated tree vertices to amalgamated tree vertices, and similarly un-amalgamated tree vertices to un-amalgamated tree vertices.
\item If $\whth_z(\phi(z)) \leq \ell \leq k$, then $\restrict \phi \ST(\ell)$ is a self-embedding of $\ST(\ell)$.
\end{enumerate}
\end{lemma}

\begin{proof}
Due to labels, tree vertices are sent to tree vertices. Moreover, an amalgamated tree vertex has been identified with a ray vertex and hence has degree 6: two neighbours as a tree vertex in a copy of $(R,r)$, one neighbour on the finite path corresponding to label 0, two neighbours as a ray vertex, and one more on the finite tree corresponding to its type. Thus  it cannot be sent to an un-amalgamated tree vertex which has only degree 3.  Now the center $z$ is itself amalgamated, thus so is $\phi(z)$, and hence $\whth_z(\phi(z) \leq k$. Thus  if $v \in \ST(k)$ is an un-amalgamated tree vertex, then $\whth_z(v)>k$, and hence  $\whth_{\phi(z)}(\phi(v))>k$.  But this implies that  $\whth_z(\phi(v))>k$ and hence $\phi(v)$ is un-amalgamated. 

For (2), assume that $\whth_z(\phi(z)) \leq \ell \leq k$. Then for any vertex $v$, $\whth_z(v) \leq \ell$ implies that   $\whth_{\phi(z)}(\phi(v))\leq \ell$, and hence    $\whth_z(\phi(v))\leq \ell$. That is  $\phi(\ST(\ell))\subseteq \ST(\ell)$. 
\end{proof}

To show that the construction is justified, we also neeed to show that the number of siblings of  $\ST(k)$ (and hence of each $\ST_s(k)$) is at most countable, and thus we seek to understand embeddings of $\ST(k)$. Recall that the spine of  $\ST(k)$ is $\Sp(k)$, and  an embedding of $\ST(k)$ induces a surjective embedding of $\Sp(k)$, and as noted in Observation \ref{spinepreservecopies} preserves ray and tree vertices as well as copies of $(R,r)$.  We now describe the exact nature of self embeddings of $\Sp(k)$, called {\em similarities}. As such,  embeddings of $\ST(k)$  induce a unique similarity on $\Sp(k)$; we will show this implies that the sign and spin functions are indeed embedded as graph properties, and this will also allow to control the number of siblings. 

\begin{definition}\label{similarity}
\begin{enumerate}
\item The \underline{fingerprint} of a path $P_{u,v}= \langle u=u_0,$ $u_1, \ldots,$ $u_n=v \rangle$ for $u,v \in \Sp(k)$ is the sequence of symbols $\langle f_0,f_1,\ldots,f_n \rangle$ such that for each $i\leq n$:
\begin{itemize}
\item $f_i = ``\sign^{R'}_{u_i}(u_{i+1}) ''   $ if $i <n$ , both $u_i$ and $u_{i+1}$ are in a copy $(R',r')$ of $(R,r)$ and $u_i$ is the first element of $R' \cap P_{u,v}$;
\item $f_i = ``\lab^{R'}(u_i)''$ if $u_i$ is in a copy $(R',r')$ of $(R,r)$, $i=n$ or $u_i$ is not the first element of $R' \cap P_{u,v}$;
\item $f_i = ``<''$ (resp. $ `` >''$) if $i <n$ , both $u_i$ and $u_{i+1}$ are ray vertices and $u_i < u_{i+1}$ (resp $u_i > u_{i+1}$)  (considered as elements of the double ray $\Z$.
\end{itemize}
\item \label{simildef} A map $\Phi = \Phi^u$ of $\Sp(k)$ is called a \underline{similarity} at the amalgamated vertex  $u \in\wR_0$ if the fingerprints of $P_{u,v}$ and $P_{\Phi(u),\Phi(v)}$ are equal for all amalgamated $v \in \Sp(k)$.
\end{enumerate}
\end{definition}

\begin{lemma}\label{similaritymap}
Let $u, v \in\wR_0$ be amalgamated vertices. Then there is a unique similarity map $\Phi=\Phi_{u,v}$ of $\Sp(k)$  such that $\Phi(u)=v$, and such a similarity map is a self embedding of $\Sp(k)$. \\
Moreover:
\begin{enumerate}
\item \label{similarityspin}
 $\wspin_{u}(w)=\wspin_v(w)=\wspin_v(\Phi(w))$ for all $w \in \wR_0^u \cap \wR_0^v$ except possibly for $w \in P_{u,v}$, and further equals 
$\wspin_u(\Phi(w))$ except possibly for $ w \in  P_{\Phi^{-1}(u),\Phi^{-1}(v)}$.
 
\item \label{similaritycolour}
$\wcol_{u}(w)=\wcol_v(w)=\wcol_u(\Phi(w)) =\wcol_v(\Phi(w))$ for all $w \in \wR_0^u \cap \wR_0^v$ except possibly those originating from a path starting from $P_{u,v}$ with strictly decreasing labels. 
\end{enumerate}
\end{lemma}

\begin{proof}
Let $u, v \in\wR_0$ be amalgamated vertices and set $\Phi(u)=v$. Thus $\whth_z(u) \leq k$ and $\whth_z(v) \leq k$, and therefore all tree vertices are amalgamated to global height at most $k$ with respect to either $z$, $u$, or $v$.  
Hence for any vertex $w$, there exists a uniqe way to define its image  $\Phi(w)$ so that  $P_{u,w}$ and $P_{\Phi(u)=v,\Phi(w)}$ have the same fingerprints; this defines the unique similarity $\Phi_{u,v}$. Note that by definition $\Phi_{u,v}$  preserves the sign function and all graph properties. 

Now for $w \in R_0^u \cap R_0^v$, $\wspin_{u}(w) = \wspin_v(\Phi(w))$ simply due to the paths having the same fingerprints. 
Further,  $\wspin_{u}(w)=\wspin_v(w)$ for all $w \notin P_{u,v}$ by Lemma \ref{globalspinpreserving}
and by the same lemma $\wspin_{u}(\Phi(w))=\wspin_v(\Phi(w))$ for all $\Phi(w) \notin P_{u,v}$, that is $w \notin P_{\Phi^{-1}(u),\Phi^{-1}(v)}$.

Similarly, by  Corollary \ref{embcolpreserv1}, $\wcol_{u}(w)=\wcol_{v}(w)= \wcol_{v}(\Phi(w)) $ except possibly those originating from a path starting from $P_{u,v}$ with strictly decreasing labels. Finally, $\wcol_{u}(w)=\wcol_{v}(\Phi(w))$ due to the paths having the same fingerprints. 
\end{proof}

We now come to the main lemma, showing that the sign function is preserved on amalgamated vertices  by graph embeddings of $\ST(k)$.

\begin{lemma}[MAIN Lemma]\label{similarityemb}
If  $\phi:\ST(k) \rightarrow \ST(k)$ is a (graph) embedding, then $ \restrict \phi \Sp(k)$ is a similarity. In particular (graph) embeddings of $\ST(k)$ preserve the sign function on amalgamated vertices.
\end{lemma}

\begin{proof}
 We have already observed that (graph) embeddings do preserve labels and the natural direction on double rays. Thus it remains to prove that a graph embedding  of $\ST(k)$ preserves the sign function on amalgamated vertices.  
  
Let $\phi$ be a  graph embedding of $\ST(k)$. We must show, without loss of generality, that if $w$ is the first element in a copy $(R',r')$ of $(R,r)$ on $P_{u,v}$ for some amalgamated vertices $u,v$, then $\sign$ is preserved at $w$. Such a $w$ is an amalgamated tree vertex (even if $w=u$), and this implies  $\whth_z(w) \leq k$ and thus  $\whth_u(w) \leq k$. Hence  $\phi(w)$ is the first element in a copy $(R'',r'')$ of $(R,r)$ on $P_{\phi(u),\phi(v)}$, $\phi(w)$ is an amalgamated tree vertex (even if $w=u$),  and thus  $\whth_z(\phi(w)) \leq k$  as well as $\whth_{\phi(u)}(\phi(w)) \leq k$.  Both $w$ and $\phi(w)$ have two neighbours in $R'$ and $R''$ respectively.  Consider target vertices $w_0, w_1 \in R'$ such that:
\begin{enumerate}
\item $w_0$ and $w_1$ are in different neighbourhoods of $w$;  
\item $P_{w,w_0}$ and $P_{w,w_1}$ have the same label sequence; 
\item \label{notonpaths} If $\widehat{w}_0$ and $\widehat{w}_1$ are the vertices in $R''$ having the same label sequences as  $w_0$ and $w_1$ from $\phi(w)$, 
then $w_0$, $w_1$, $\widehat{w}_0$ and $\widehat{w}_1$  are not on paths of decreasing labels from $P_{u,z}$, $P_{u,\phi(u)}$, or $P_{z,\phi(u)}$;
\item $\whth_u(w_i)=\wcol_u(w_i)=\whth_z(w_i)=\wcol_z(w_i)=k$ for each $i$.
\end{enumerate}

\begin{figure}[ht]
\includegraphics[width=.8\textwidth]{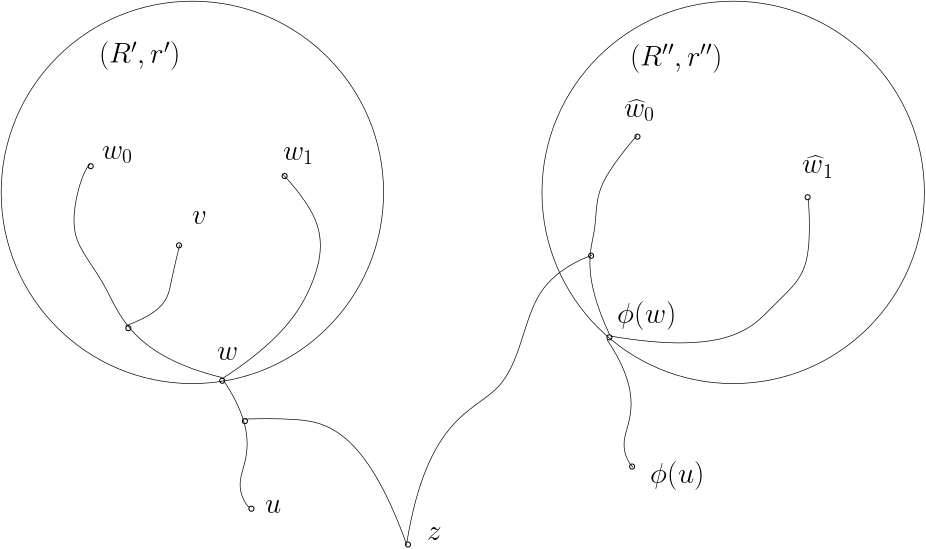}\label{fig:simi}
\caption{Lemma \ref{similarityemb}: Spin preserving.}
\end{figure}

This can be accomplished by choosing $w_i$ along paths formed by $\ell$ concatenated paths of unimodal labels $\langle 012\cdots (k-1)kk(k-1)\cdots 10\rangle$ starting at $w$ in different neighbourhoods of $w$ for some $\ell$ large enough to satisfy item \ref{notonpaths}. 

Wlog $\spin_w(w_0)=-1$ and hence $\spin_w(w_1)=+1$. But $\wspin_z(w_0)=\wspin_{u}(w_0)=\spin_w(w_0)=-1$ by Lemma \ref{globalspinpreserving} since $w_0 \notin P_{u,z}$, and similarly $\wspin_z(w_1)=\wspin_{u}(w_1)=\spin_w(w_1)=+1$. This means $w_0$ is amalgamated to the centre of a copy of the double ray $\SD_0$ (from $\ST(k-1)$), and $w_1$ to the centre of a double ray (from $T_k$) that cannot embed into $\SD_0$ preserving their centres, and vice-versa. 
The corresponding vertices $\widehat{w}_0, \widehat{w}_1  \in R''$ are those starting at $\phi(w)$ with the same label sequence. 
But $\wspin_z(\widehat{w}_i)=\spin_{\phi(w)}(\widehat{w}_i)$ since $\widehat{w}_i \notin P_{z,\phi(w)}$. Thus wlog $\wspin_z(\widehat{w}_0)=-1$ and hence $\widehat{w}_0$ is amalgamated to the centre of a copy of the double ray $\SD_0$, and $\widehat{w}_1$ to the centre of a double ray that cannot embed into $\SD_0$ preserving their centres, and vice-versa. Hence there is no alternative but $\phi$ sending $w_0$ to $\widehat{w}_0$ and similarly $w_1$ to $\widehat{w}_1$. 
But this means that $\sign$ is preserved at $w$ as desired. 
\end{proof} 

We have already observed that any sibling of  $\ST(k)$ contains (a copy of) $\Sp(k)$, and hence the above result immediately carries to siblings of $\ST(k)$.

\begin{corollary}\label{similarityembsibling}
If  $\SSS$ and $\SSS'$ are siblings of $\ST(k)$, then any embedding $\phi: \SSS \rightarrow \SSS'$ induces a similarity on $\Sp(k)$. 
\end{corollary}

\begin{proof}
Let $\SSS$ and $\SSS'$ be siblings of $\ST(k)$, and $\phi: \SSS \rightarrow \SSS'$ an embedding. Now since $\SSS$ is a sibling,  let $\psi: \ST(k)  \rightarrow \SSS$  be an embedding, and we may consider $\SSS'$ as a substructure of  $\ST(k)$. Hence  $ \phi \circ \psi: \ST(k)  \rightarrow  \ST(k) $ is an embedding whose restriction to $\Sp(k)$ is a similarity by Lemma \ref{similarityemb}. But $ \restrict \psi \Sp(k)$ is itself a (surjective) similarity, and hence so is $\psi^{-1}$. Thus $\restrict \phi \Sp(k) = \restrict {\phi\circ\psi\circ\psi^{-1}} \Sp(k) $ is a similarity.
\end{proof}

Thus any graph embedding of $\ST(k)$  (or of any sibling)  is a similarity on $\Sp(k)$, and conversely we will see later how to use particular similarities of  $\Sp(k)$ to create embeddings of $\ST(k)$, meaning how to correctly match the type assignments of ray vertices.  As a corollary to Lemma \ref{similarityemb} we will need the corresponding property for all siblings of $\ST(k)$. Before that, we first show that  the type assignments on double rays can only disagree with the image of an embedding of $\ST(k)$ for only finitely many ray vertices. That is, only finitely many ray vertices of type $0$ are mapped to type $1$ ray vertices, or if we recall that  type assignments are  implemented though finite trees attached to those ray vertices, we show that $\ST(k) \setminus \phi(\ST(k))$ is finite for any embedding $\phi$ of $\ST(k)$. 

There are obvious proper self-embeddings  of $\ST(k)$ (and each $\ST_s(k)$),  namely any translation along the double ray $\SD_0$,  so that all type 1 ray vertices are mapped into type 1 ray vertices. Indeed by construction, all trees attached to ray vertices on  $\SD_0$  are identical, and thus can be mapped (isomorphically) to the corresponding tree by translation, and hence $\ST(k) \setminus \phi(\ST(k))  $ is finite for such embeddings due to finitely many type 0 ray vertices mapped to type 1 ray vertices. Thus $\ST(k)$ is almost equal to its image by a translation embedding. We show that this situation occurs for all embeddings, that is $\ST(k) \setminus \phi(\ST(k))  $   is finite for all embeddings. It is worth observing that this property propagates to all siblings. 

\begin{lemma}\label{embfinite}
$\ST_s(k) \setminus \phi(\ST_s(k))$ is finite for any self-embedding $\phi$ and $s < \sss$. 
The only possible difference is in a finite number of ray vertices of different type assignments. 
\end{lemma}

\begin{proof}
Since each $\ST_s(k)$ is a sibling of $\ST(k) $, it suffices to prove it for the latter. 

We proceed by induction on $k$. It is easily verified for $k=0$ since proper embeddings of $\ST(0)$ consist of (proper) translations of $\SD_0$ (in its natural direction). We then assume the statement is true for all $i<k$, and we fix a self-embedding $\phi$ of $\ST(k)$. 

By Lemma \ref{similarityemb}, $\phi$ induces a surjective similarity on $\Sp(k)$, and hence preserves preserves labels, copies of $(R,r)$, amalgamated vertices, ray vertices, the natural direction of double rays, and the sign function on amalgamated vertices. Thus any vertex in $\ST(k) \setminus \phi(\ST(k))$ must  come from ray vertices of type 0 assignment  being mapped to type 1 assignments, and these type assignments are set by the amalgamations during the construction. 

Since $\ST(k)$ is the disjoint union of craters of target vertices, we will show that only finitely many such craters may differ from their image, and that the difference is finite in all those cases where they differ. It suffices to consider craters of amalgamated vertices, since otherwise the crater of an un-amalgamated target vertex consists only of its spine elements and is mapped to a crater to an un-amalgamated vertex and thus the embedding is surjective in that case. 

Consider an amalgamated  target vertex $u$ such that $\whth_z(u)=\wcol_z(u)=\ell$, and hence $\ell\leq k$. Then $u$ was activated during the construction of $\ST(\ell)$, and was amalgamated with the centre of a tree $T$ (a copy of $\ST(\ell-1)$ or some $T_{\ell}$). 

First consider the case where $\whth_z(\phi(u))=\ell$. Then observe that $\phi(u)$ must also be a target vertex. This can be shown as follows: 
since $\phi$ induces a surjective similarity on $\Sp(k)$ and thus $\phi^{-1}(z)$ exists; thus let $w$ be such that $P_{z,u} \cap P_{\phi^{-1}(z),u}=P_{w,u}$. Then either the last consecutive pair on $P_{z,u}$ (of label $\ell$) is on $P_{w,u}$, in which case it is also on $P_{\phi(w),\phi(u)}$, or else $P_{w,u}$ consists of decreasing labels strictly less than $\ell$ and in which case $P_{z,\phi(w)}$ must contain a consecutive pair of labels $\ell$. 
Thus, if $\wspin_z(u)= \wspin_z(\phi(u))$, then the same tree $T$ is amalgamated to $u$ and $\phi(u)$ at their centre, and $\phi$ induces an embedding of that tree, sending the crater at $u$ to the crater at $\phi(u)$; by uniqueness of the similarity (by Lemma \ref{similaritymap}), $\phi$ induces an isomorphism of $T$ (essentially the identity) and $T \setminus \phi(T) = \emptyset$. If instead  $\wspin_z(u) \neq \wspin_z(\phi(u))$, then the tree $T'$ amalgamated at $\phi(u)$ is different than $T$, but  $T' \setminus \phi(T)$ is finite by the induction hypothesis (note that $\ST(\ell-1)$ or some $T_{\ell}$ are siblings), and this can occur only finitely many times by Corollary \ref{embcolpres2}.

\begin{figure}[ht]
\includegraphics[width=.6\textwidth]{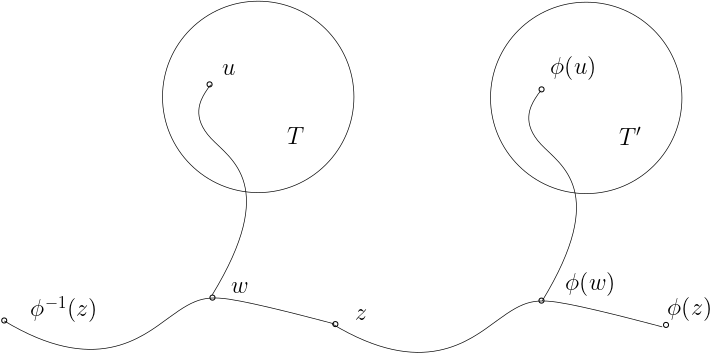}\label{fig:embfinite1}
\caption{Lemma \ref{embfinite}: The case  $\whth_z(\phi(u)) = \whth_z(u)=\wcol_z(u)=\ell$. }
\end{figure}

Now assume that $\whth_z(\phi(u))=m>\ell$. This means that the image of $T_{\ell}$ by $\phi$ was created at the later stage $m$  in the construction, 
and so the vertex $\phi(u)$ is part of a tree $T'$ (a copy of $\ST(m-1)$ or some $T_{m}$) that was amalgamated with its centre to a target vertex $v$ during the construction of $\ST(m)$, thus $\whth_z(v)=\wcol_z(v)=m$, and $\whth_{v}(\phi(u)) <m$. Hence $\whth_{\phi(z)}(v)<m$, and   the preimage $w=\phi^{-1}(v)$ satisfies $\whth_z(w)<m$. But this means that $\phi$ induces an embedding of $\ST(m)$ into $T'$, and the induction hypothesis ensures that $T' \setminus \phi( \ST(m))$ is finite. Note this is simply the case of the $m$-crater containing $\phi(z)$ where $m = \whth_z(\phi(z))$. 

\begin{figure}[ht]
\includegraphics[width=.6\textwidth]{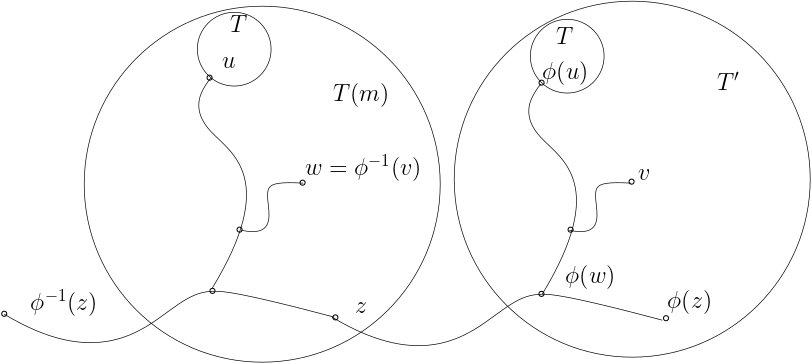}\label{fig:embfinite2}
\caption{Lemma \ref{embfinite}: The case  $\whth_z(\phi(u)) =m > \whth_z(u)=\wcol_z(u)=\ell$. }
\end{figure}

Finally assume that $\whth_z(\phi(u))=m<\ell$. Thus if $\whth_u(w)<\ell$, then $\whth_{\phi(u)}(\phi(w))<\ell$ and thus 
$\whth_z(\phi(w))<\ell$. This means that $\phi$ induces an embedding of $T$ into $\ST(\ell-1)$, and the induction hypothesis again ensures that $\ST(\ell -1) \setminus \phi(T) $ is finite. Note this is simply the case of the image of the $\ell$-crater containing $z$ (as the image of $\phi^{-1}(z)$ where $\ell = \whth_z(\phi(z))$. 

\begin{figure}[ht]
\includegraphics[width=.6\textwidth]{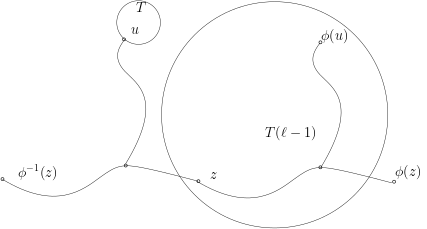}\label{fig:embfinite3}
\caption{Lemma \ref{embfinite}: The case  $\whth_z(\phi(u)) =m < \whth_z(u)=\wcol_z(u)=\ell$. }
\end{figure}

This completes the proof. 
\end{proof} 

We can now verify two requirements of the construction. First we show that each non-isomorphic sibling of each $\ST(k)$ contains a double ray
non-isomorpic to $\SD_0$, that is having a type $1$ ray vertex followed by a vertex of type $0$.

\begin{corollary}\label{nonisop0}
Every non-isomorphic sibling of each $\ST(k)$ contains a double ray non-isomorpic to $\SD_0$, that is having a type $1$ ray vertex followed by a vertex of type $0$. 
\end{corollary}

\begin{proof}
The tree $\ST(1)$ already contains infinitely many such double rays (from embeddings of $S_{0,0}$), and by Lemma \ref{embfinite}, any sibling of $\ST(k)$ contains all but finbitely many of those rays. 
\end{proof} 

We can also show that each $\ST_s(k)$ has countably many siblings.

\begin{corollary}\label{ctblsiblings}
$\sib(\ST_s(k)) = \aleph_0$ for all $k$.
\end{corollary}

\begin{proof}
By construction each $\ST_s(k)$ is countable, and  this also follows from simply being locally finite trees. Thus for any amalgamated vertex $v$, all self-embeddings $\phi$ such that $\phi(z)=v$ agree on  all vertices in its spine $\Sp(k)$ of global height at most $k$ by Lemmas \ref{similaritymap} and \ref{similarityemb}.  Now by Lemma \ref{embfinite}, $\ST(k) \setminus \phi(\ST(k))$ is finite. So there can be only finitely many siblings $\phi(\ST(k)) \subseteq S \subseteq \ST(k)$. Hence $\sib(\ST(k))=\sib(\ST_s(k))\leq \aleph_0$ for all $k$.

We already noticed that each $\ST(k)$ has infinitely many siblings due to translations along $\SD_s$, hence $\sib(\ST_s(k)) = \aleph_0$  exactly.
\end{proof} 
  
We show that all $\ST_s(k)$ are pairwise non-isomorphic siblings at every stage, and in fact we prove a bit more so to support the induction argument.  
    
\begin{lemma}\label{basenonisomorphism}
For each $s \neq s' < \sss$ and $k \in \N$,  $\ST_s(k) \not\cong \ST_{s'}(k)$. \\
Moreover, if $k+1=2^i(2j+1)$, and $S_{i,j}$ as a substructure of $\ST_s(i)$ and $\ST_s(k)$ was expanded to $S_{k+1}$ following the inductive construction so that all tree vertices in $T_{k+1}$ are amalgamated up to global height $k$ with respect to its centre $c$, then 
 $S_{k+1} \not\cong \ST_{s'}(k)$. 
\end{lemma}   

\begin{proof} 
 We proceed by induction on $k$.  We have already noted that the trees $\ST_s(0)$ are pairwise non-isomorphic (since the only paths involved are not isomorphic). Moreover since $1=2^0(2\cdot 0+1)$,  $S_1=S_{0,0}$ and by definition is not isomorphic to any $\ST_s(0)$.
 
Let $k>0$ be the smallest counterexample, and suppose first that $\phi: \ST_{s'}(k) \rightarrow \ST_s(k)$ is an isomorphism with $s' \neq s$. Let $v=\phi(z_{s'})$, then $v$ must be amalgamated (since $z_{s'}$ is).  Let $\ell =\whth_{z_s}(v) \leq k $. If $\ell=0$, then this means that $\phi(z_{s'}) \in \SD_s$,  but this is impossible by construction since the double rays $\SD_{s'}$ and $\SD_s$ are not isomorphic. 

Thus $\ell>0$, and the inductive definiton applies. Write $\ell= 2^i(2j+1)$. According to the cases of the constructions, $v$ is either contained in a copy of  $\ST(\ell-1)$ (=$\ST_0(\ell-1)$ by construction), or of an extended  sibling $S_{\ell}$ of $S_{i,j}$  that was inserted in the tree by amalgamation identifying its centre to a target vertex . But then either $\restrict \phi  \ST_{s'}(\ell-1): \ST_{s'}(\ell-1) \rightarrow \ST(\ell-1)$ or  $\restrict \phi  \ST_{s'}(\ell-1): \ST_{s'}(\ell-1) \rightarrow S_{\ell}$  is an isomorphism, a contradiction in either case to the induction hypothesis. 
  
\begin{figure}[ht]
\includegraphics[width=.6\textwidth]{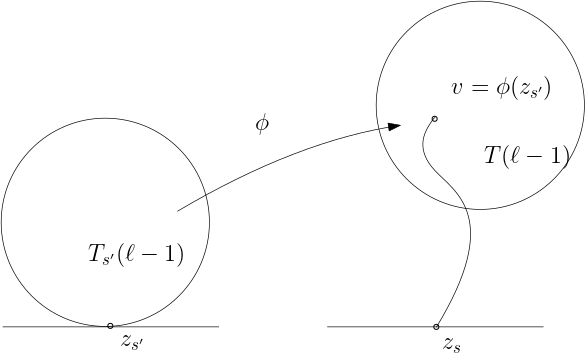}\label{fig:basenonisomorphism1}
\caption{Lemma \ref{basenonisomorphism}: The case that $\phi: \ST_{s'}(k) \rightarrow \ST_s(k)$ is an isomorphism.}
\end{figure}
 
Otherwise write $k+1=2^i(2j+1)$, and consider $S_{k+1}$ expanded from $S_{i,j}$ so that all tree vertices in $S_{k+1}$ are amalgamated up to global height $k$ with respect to its center $c$. Assuming that $\phi: \ST_{s'}(k) \rightarrow S_{k+1}$ is an isomorphism, let $v=\phi(z_{s'})$, and $\ell=\whth_c(v)$.   If $\ell \leq i$, then $v$ was amalgamated in $S_{i,j}$ and  $\restrict \phi  \ST_{s'}(i): \ST_{s'}(i) \rightarrow S_{i,j}$  is an isomorphism by Lemma \ref{embedkcopies}, a contradiction as $S_{i,j}$ was specifically chosen non-isomorphic to $\ST_{s'}(i)$ (and any $\ST_s(i)$). If $\ell > i$, then $v$ was amalgamated following the inductive construction by assumption, and hence is part of an amalgamation of some $\ST_s(m)$ or a $S_{m+1}$ for some $i \leq m<\ell$. But then either $\restrict \phi  \ST_{s'}(m): \ST_{s'}(m) \rightarrow \ST_s(m)$  or $\restrict \phi  \ST_{s'}(m): \ST_{s'}(m) \rightarrow S_{m+1}$ is an isomorphism by Lemma \ref{embedkcopies}, again a contradiction in either case. 
 
\begin{figure}[ht]
\includegraphics[width=.6\textwidth]{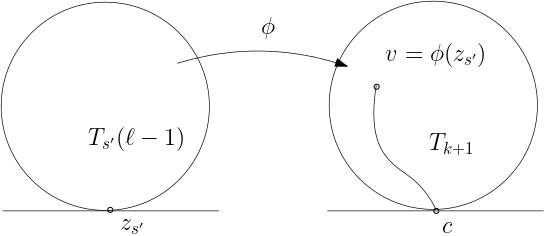}\label{fig:basenonisomorphism2}
\caption{Lemma \ref{basenonisomorphism}: The case that $\phi: \ST_{s'}(k) \rightarrow  S_{k+1}$  is an isomorphism.}
\end{figure}
\end{proof} 
 
\subsection{The trees $\langle \ST_s: s < \sss \rangle$}

We are now ready to define  $\langle \ST_s: s < \sss \rangle$.

\begin{definition}
Define $\ST_s= \bigcup_k\ST_s(k)$ for each $s < \sss$.
\end{definition}

Recall the spine $\Sp$ defined in Definition \ref{spine}. Since $\Sp(k)$ is the spine of each $\ST_s(k)$, then clearly $\Sp$ (up to isomorphism) is the spine of each $\ST_s$. Moreover we have the following result arising from known properties of $\ST(k)$.

\begin{lemma}\label{embedT}
Let $\phi$ be a self-embedding of $\ST$ into $\ST$, then:
\begin{enumerate}
\item  $\phi$ induces a similarity on $\Sp$.
\item If $\whth_z(\phi(z))\leq k$, $\restrict \phi \ST(k)$ is a self-embedding of $\ST(k)$.
\item $\ST \setminus \phi(\ST)$ is finite. The only possible difference is in a finite number of ray vertices of different type assignments. 
\end{enumerate}
\end{lemma}

Note that in the last case a type 0 ray vertex being mapped to a type 1 means that a single leaf is not in the image, this will be useful later. 

We first confirm that we have at least $\sss$ siblings.

\begin{proposition}\label{atleast2sib}
$\ST_s \not\cong \ST_{s'} $ for any $s \neq s' < \sss$.
\end{proposition}
  
\begin{proof}  
 Suppose that $\phi:\ST_{s'} \rightarrow \ST_s$ is an isomorphism, and let $k=\whth_{z_s}(\phi(z_{s'})$. Then $\restrict \phi \ST_{s'}(k): \ST_{s'}(k) \rightarrow \ST_s(k)$ is an isomoprhism, contradicting Lemma \ref{basenonisomorphism}.
 \end{proof} 

And finally we are ready to conclude that $\ST$ as  exactly $\sss$ siblings. 
    
\begin{proposition}\label{atmost2sib}
If $\SSS \sim \ST$, then $\SSS \cong \ST_s$ for some  $s<\sss$.
\end{proposition}
  
\begin{proof}  
Let $\SSS$ be a sibling of $\ST$. Being all siblings, we may assume that $\ST=\ST_0 \supseteq \ST_s \supseteq \SSS \supseteq \theta(\ST)$ for all $s < \sss$ and some  self-embedding $\theta:\ST \rightarrow \ST$.  By  Lemma \ref{embedT}, we can find  $n$  such that:
\begin{enumerate}
\item \label{boundn} $z_s \in \ST(n)$  for each $s<\sss$, 
\item  $\theta(z_0) \in \ST(n)$,
\item $\ST \setminus \theta(\ST) \subseteq \ST(n)$.
\end{enumerate}

Define  $\SSS(n) = \SSS \cap \ST(n)$. So $\theta(\ST(n) \subseteq \SSS(n) \subseteq \ST(n)$, and hence $\SSS(n)$ is a sibling of $\ST(n)$. Thus by construction  either $\SSS(n) \cong \ST_s(n)$ for some $s < \sss$,  or else $\SSS(n) \cong S:=S_{n,j}$  for some $j$. 

First assume the latter, and we shall show that $\SSS \cong \ST$. We do so by  extending an isomorphism from $S$ to $\SSS(n)$ mapping craters to craters attached to target vertices of the same spin, and hence have the same type assignments on all ray vertices of those craters, producing the required isomorphism from $\ST$ to $\SSS$.
Let $c$ be the centre of $S$ and  $k=2^n(2j+1)$. Now call $u$ the  tree vertex of $\ST(k)$ at the end of the path $P_{z,u}$ starting at $z=z_0$ with unimodal labels $\langle 01\cdots kk \cdots 10 \rangle$ increasing to $k$ and back to $0$, and such that  $sign_z(u)=-1$ and thus $\spin_{z}(u)=+1$. Note that $u$ and $z$ are in the same copy $(R',r')$ of $(R,r)$, and $u$ is a target vertex of (global) height $k$ (with respect to $z$). Thus at stage $k$ of the construction of $\ST(k)$, $S$ was extended (following the inductive construction) to $S'$  so that all tree vertices in $S'$ are amalgamated up to global height $k-1$ (with respect to $c$) before $(S',c)$ was amalgamated with $(\ST(k-1),u)$ over $(R',r')$. In particular we  consider $S'$ (and $S$) as a substructure of $\ST(k)$ with $c$ identified with $u$. 

Fix an isomorphism $\phi: S \rightarrow  \SSS(n)$, and let $\phi(u)=v \in \SSS(n)$. By Corollary \ref{similarityembsibling}, $\phi$ induces a similarity map $\hat{\Phi}=\restrict \phi \Sp(n)$ on a copy $\hat{\Sp}(n)$ of $\Sp(n)$ as a substructure of $S$ (and $\ST(n)$ centred at $u$), and  by Lemma \ref{similaritymap} it is the unique similarity from $S$ to $\SSS(n)$ sending  $u$ to $v$.  
Now from the global point of view of $\ST$, there is also a similarity map  $\Phi=\Phi^u_{u,v}$ on $\Sp$ (as a substructure of $\ST$), and it turns out that $\hat{\Phi}$ and $\Phi$ agree on  $\hat{\Sp}(n)$.
The reason is that, by definition, both maps  preserve the fingerprint of a path $P_{u,w}$ for $w \in S$. But note that those fingerprints may be different from the point of view of $S$ and $\ST$: of course the labelling and ray orderings agree, but the sign values are computed on one hand from the point of view of $S$ (and its centre before being embedded into $\ST$), and on the other hand from the point of view of $\ST$ and its centre $z$. Yet the fact that  $\hat{\Phi}$ and  $\Phi$ preserve the corresponding fingerprints implies that both will agree on their image of $w$. 
Hence we must find the required embedding from $\ST$ to $\SSS$ by extending $\hat{\Phi}$ to $\Phi$: we will show that $\Phi$ preserves the spin  at target vertices of all craters outside $S$.  To do so, we will go through $z$ by considering  $\Phi_{u,v}=\Phi_{z,v} \circ \Phi_{u,z}$, 

By Lemma \ref{similaritymap}, $\wspin_u(w)=\wspin_z(w)=\wspin_z(\Phi_{u,z}(w))$ for all $w \in \wR_0^u \cap \wR_0^z$, except possibly for $w \in P_{u,z}$. 
This means that for all target vertices $w$ of global height $\ell > n$ (with respect to $u$) not on  $P_{u,z}$, $\Phi_{u,z}(w)$ is a target vertex (since $\whth_z(v) \leq n$) of the same spin (with respect to $z$), and thus the same tree is amalgamated at $w$ and $\Phi_{u,z}(w)$. Hence $\Phi_{u,z}$ can be extended to an isomorphism  $\phi^w_{u,z}$ (matching the type assignments at ray vertices) on th
e $\ell$-crater centered at $w$ to the $\ell$-crater centred at $\Phi_{u,z}(w)$.
Note this includes all target vertices in $S' \setminus S$ which were amalgamated before $S'$ was itself amalgamated to $\ST(k)$.
Now $w=z$ is the only (target) tree vertex on $P_{u,z}$ (since $\Phi(u)=v$ has already been taken care of). By Definition  \ref{spinsign}, $\sign_u(z)=\spin_{z}(u) = +1$ implying  through similarity that $\sign_z(\Phi_{u,z}(z))=+1$, and thus $\wspin_z(\Phi_{u,z}(z))=-1$.  Hence a copy of $\ST(k-1)$ is amalgamated at ($z$ and) $\Phi_{u,z}(z)$, and $\Phi_{u,z}$ can again be extended to an isomorphism $\phi^z_{u,z}$ on the $k$-crater centered at $w=z$ (namely $(\ST(k-1),z)$) to the $k$-crater centered at $\Phi_{u,z}(w)$ (a copy of $(\ST(k-1),z)$.
Note this is a crucial part to ensure that $\SSS \cong \ST=\ST_0$, and not any other  $\ST_s$. 
\begin{figure}[ht]
\includegraphics[width=.8\textwidth]{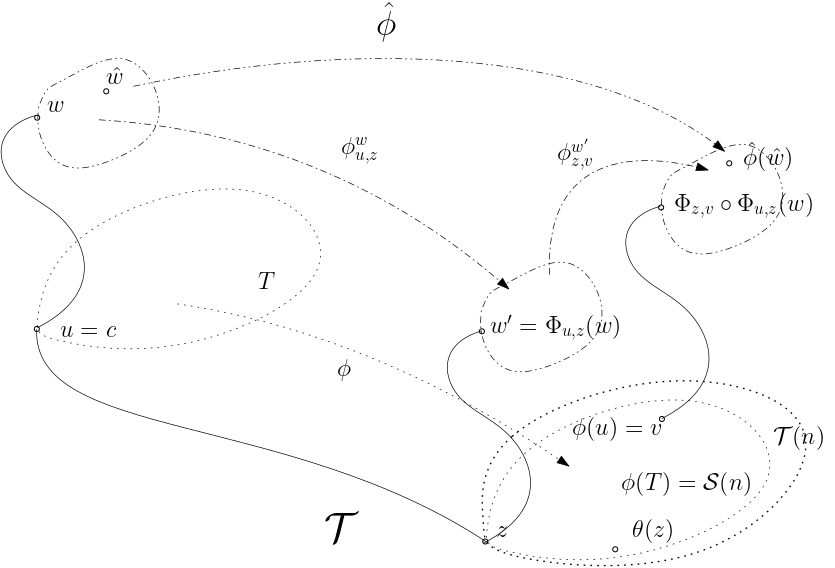}
\caption{Proposition \ref{atmost2sib}: The case $\SSS(n)\cong T=S_{n,j}$.}
\end{figure}
Next consider  the similarity map $\Phi_{z,v}$ on $\Sp$. Then again, by Lemma \ref{similaritymap}, $\wspin_z(w')=\wspin_z(\Phi_{z,v}(w'))$ for all $w' \in \wR_0^v \cap \wR_0^z$  with the only possible exception of vertices $w' \in P_{z,v}$. Thus, if $w'$ is a target vertex of global height $\ell >n$ with respect to $z$, then  $w' \notin P_{z,v}$  (since $\whth_z(v)\leq n$), and  
$\Phi_{z,v}(w')$ is again a target vertex and $\wspin_z(w')=\wspin_z(\Phi_{z,v}(w'))$. Thus  the same tree is amalgamated at $w'$ and $\Phi_{z,v}(w')$, and $\Phi_{u,z}$ can be extended to an isomorphism  $\phi^{w'}_{z,v}$   (matching the type assignments at ray vertices) on the $\ell$-crater centered at $w'$ to the $\ell$-crater centered at $\Phi_{z,v}(w')$.

Combining the maps, the required isomorphism $\hat{\phi}: \ST \rightarrow \SSS$ can be summarized as follows. For $\hat{w} \in \ST$:\\
\begin{tabular}{lll}
$\hat{\phi}(\hat{w})$ & $= \phi(\hat{w})$ & if $\hat{w} \in T$ (equivalently $\whth_u(\hat{w})\leq n$); \\
 & $=  \phi^{w'}_{z,v}  \circ \phi^w_{u,z} (\hat{w})$ & if $\hat{w} \notin T$, equivalently $\ell= \whth_u(\hat{w}) > n$, \\
 & & $\hat{w}$ belongs to the $\ell$-crater of the target vertex $w$,\\
 & & and $w'= \Phi_{u,z}(w)$.\\
\end{tabular} 

\medskip 

The case that $\SSS(n) \cong \ST_s(n)$ for some $s < \sss$ is similar but relatively simpler; we show that $\ST_s \cong \SSS $ . 

Fix an isomorphism $\phi: \ST_s(n) \rightarrow  \SSS(n)$, and let $v=\phi(z_s) \in \SSS(n) \subseteq \ST(n)$. 
By Corollary \ref{similarityembsibling}, $\phi$ induces a similarity map $\hat{\Phi}=\restrict \phi \Sp(n) $. By Lemma \ref{similaritymap}, it is the unique similarity from $\Sp(n)$  sending  $z_s$ to $v$, and must therefore readily agree with the similarity $\Phi=\Phi_{z_s,v}$ on $\Sp$. We will show that $\Phi$ preserves the spin at target vertices of all craters outside $\ST_s(n)$. Note here that the spin in $\ST_s$ is determined by its centre $z_s$, and the spin in $\SSS \subseteq \ST$  is determined by the centre $z_0$.
By Lemma \ref{similaritymap},  $\wspin_{z_s}(w)=\wspin_v(\Phi(w))$  for all $w \in \wR_0^u \cap \wR_0^v$, except possibly for $w \in P_{z_s,v}$. 
Moreover, by Lemma \ref{globalspinpreserving}, $\wspin_v(\Phi(w))=\wspin_{z_0}(\Phi(w))$ as long as $\Phi(w) \notin P_{v,z_0}$. 
But any target vertex $w \notin \ST_s(n)$ had global height $\ell>n$ and satisfies $\whth_{z_s}(w) > n$, hence $w$ cannot be on $ P_{z_s,v} \subseteq \ST(n)$, 
$(\Phi(w)$  cannot be on $ P_{v,z_0} \subseteq \ST(n)$ since $\whth_{z_s}(w) > n$ implies that $\whth_{v}(\Phi(w)) > n$, 
and $\Phi(w)$ is also a target vertex since moreover  $\wcol_{z_s}(w)=\wcol_v(\Phi(w))$ by Lemma \ref{similaritymap}, and $\wcol_z(\Phi(w))=\wcol_v(\Phi(w))$ by Lemma \ref{globcolpreserv}.
Hence $\Phi$ can be extended to an isomorphism  $\phi^w$ (matching the type assignments at ray vertices) on the $\ell$-crater centered at $w$ to the $\ell$-crater centered at $\Phi(w)$.
\begin{figure}[ht]
\includegraphics[width=.8\textwidth]{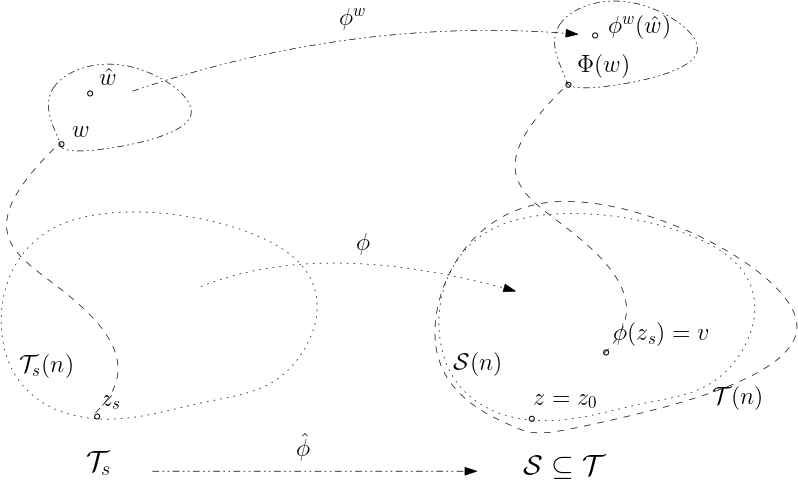}
\caption{Proposition \ref{atmost2sib}: The case  $\SSS(n) \cong \ST_s(n)$.}
\end{figure}

Combining the maps, the required isomorphism $\hat{\phi}: \ST_s \rightarrow \SSS$ can be summarized as follows. For $\hat{w} \in \ST_s$:\\
\begin{tabular}{lll}
$\hat{\phi}(\hat{w})$ & $= \phi(\hat{w})$ & if $\hat{w} \in  \ST_s(n)$ (equivalently $\whth_{z_s}(\hat{w})\leq n$); \\
 & $=  \phi^{w} (\hat{w})$ & if $\hat{w} \notin T$, equivalently $\ell= \whth_u(\hat{w}) > n$, \\
 & & $\hat{w}$ belongs to the $\ell$-crater of the target vertex $w$.\\
\end{tabular}

This completes the proof. 
\end{proof} 
 
 We finally have all the ingredients to prove our first main theorem. \\
 
\noindent {\bf Theorem \ref{thm:main1}.}
For each non-zero $\sss \in \N$, there is a locally finite tree $\ST$ with exactly $\sss$ siblings, considered either as relational structures or trees. Moreover, for $\sss=1$, the tree is not a ray, yet it has a non-surjective embedding. \\
Thus the conjectures of Bonato-Tardif, Thomass\'e, and Tyomkyn regarding the sibling number of trees and relational structures are all false. 
  
\begin{proof}
By Propositions \ref{atleast2sib} and \ref{atmost2sib}, $\ST$ is a locally finite tree with $\sib(\ST)=\sss$, hence disproving the  Bonato-Tardif conjecture in the case $\sss \geq 2$, and hence also Tyomkyn's first conjecture. 

By Lemma \ref{embedT}, any sibling of $\ST=\ST_0$ viewed as a substructure of $\ST$  differs from $\ST$ by a finite set of ray vertices of different type assignments, meaning a single leaf has been removed  from finitely many type 1 vertices to become type 0 vertices. But this means that $\ST \oplus 1$ does not embed in $\ST$, and hence any sibling of $\ST$, viewed as a binary relational structure, is connected and thus a tree. In this case Thomass\'e's conjecture is equivalent to Bonato-Tardif's conjecture, and thus also false.

Finally consider the  special case $\sss=1$ so that $\sib(\ST)=1$. The embedding $\phi$ of $\SD_0$ given by translation $\phi(v_i)=v_{i+1}$, and its natural extension to $\ST$, is a proper embedding ($v_1$ of type 1 has become type 0), and $\ST$ is certainly not a ray. Thus in this case $\ST$ disproves Tyomkyn's second conjecture. 
\end{proof}

\section{Siblings of Partial Orders}
The above construction of locally finite trees with prescribed finite number of siblings can be adapted to provide a similar construction of partial orders with a prescribed finite number of sibling, hence our second main theorem. \\

\noindent {\bf Theorem \ref{thm:main2}.}
For each non-zero $\sss \in \N$, there is a partial order  $\SP$ with exactly $\sss$ siblings (up to isomorphy). \\

We briefly outline the main ingredients for the proof. This will be done by following the construction using the modified double rays $\SD'$ as described earlier (see Figure \ref{doubleraysposets}), to obtain locally finite trees $\ST'$ with a prescribed finte number of siblings. Then a partial order is defined on $\ST'$ to create  a partial order $\SP$ so to have the same monoid of embeddings (either as a tree or as a partial order), and hence the result follows. 

\subsection{Partial ordering on the gadgets}
In the construction of $\ST$ (and $\ST'$), we have used various gadgets of the form $PK(2n,m)$, and we define a partial order on $PK(2n,m)$  in the form of a fence as follows.

\begin{definition}\label{gadget}
On  $PK(2n,m)$, the finite gadget formed by connecting $u \in K_{1,m}=(u,V)$ to the end vertex $u_{2n}$ of a path $\langle u_0, u_1, \ldots, u_{2n} \rangle$ of length $2n$, define a partial order in the form of a fence as follows:\\
\begin{enumerate}
\item $ u_{0} < u_{1}$, $u_{2n}< u_{2n-1},$ 
\item $ u_{2i} < u_{2i-1}$ and $u_{2i}<u_{2i+1} $  for  $0<i<n$, 
\item $ u_{2n} < v $  for all $v \in V$.
\end{enumerate}
\end{definition}

\begin{figure}[ht]
\includegraphics[width=.8\textwidth]{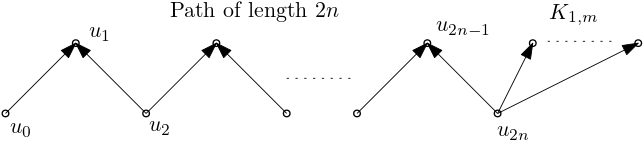}
\caption{The gadget  $PK(2n,m)$ as a partial order.}
\end{figure}

We record the following immediate observation. 

\begin{observation}\label{obs:gadgetemb}
$PK(2n,m)$ embeds into $PK(2n',m')$ as rooted trees if and only if $n=n'$ and $m \leq m'$.\\
Moreover any graph embedding of such a gadget into another one as a rooted tree is an order embedding, and vice-versa.
\end{observation}

\subsection{Partial ordering on $(R,r)$}
To define a partial ordering on $(R,r)$, we will need the following property.

\begin{lemma}\label{lem:unisign}
Given a path $P_{w_1,w_2}$  in $(R,r)$ between tree vertices $w_1, w_2$ with unimodal labels $\langle 0 1 \cdots k k \cdots 1 0 \rangle$ increasing to $k>0$ and back to $0$, then 
$\sign_{w_1}(w_2)=-\sign_{w_2}(w_1)$.
\end{lemma}

\begin{proof}
Let $w$ so that $P_{r,w_1} \cap P_{r,w_2}=P_{r,w}$, and we may assume without loss of generality that $w \neq w_1$. 

\begin{figure}[ht]
\includegraphics[width=.5\textwidth]{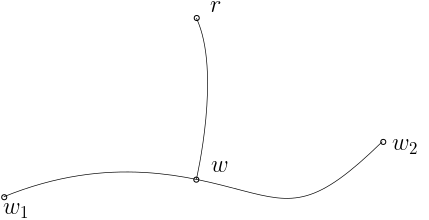}
\caption{Lemma \ref{lem:unisign}: $P_{r,w_1} \cap P_{r,w_2}=P_{r,w}$.}
\end{figure}

First suppose that $w=w_2$, then $w_1$ and $r$ belong to opposite neighbourhoods of $w_2$, and hence $\sign_{w_2}(w_1)=-\spin_r(w_2)$ according to Definition \ref{spinsign}.
Further in this case, $w_2$ and $r$ belong to the same neighbourhood of $w_1$, and hence $\sign_{w_1}(w_2)=+\spin_r(w_1)$ again according to Definition \ref{spinsign}. But $\spin_r(w_1)= \spin_r(w_2)$ since the path  $P_{w_1,w_2}$ contains a single consecutive pair and one more tree vertex beyond $w_2$. Hence  $\sign_{w_1}(w_2)=-\sign_{w_2}(w_1)$ in this case.

Now assume that $w \neq w_2,w_1$. Then $w_1$ and $r$ belong to the same neighbourhoods of $w_2$, and similarly 
$w_2$ and $r$ belong to the same  neighbourhoods of $w_1$. Hence 
$\sign_{w_2}(w_1)=\spin_r(w_2)$ and  $\sign_{w_1}(w_2)=\spin_r(w_1)$ by Definition \ref{spinsign}.
But the single consecutive pair of $P_{w_1,w_2}$ will be counted exactly once in either $\spin_r(w_2)$ or $\spin_r(w_1)$, and hence since the number of tree vertices is the same in both cases we get $\spin_r(w_2)=-\spin_r(w_1)$. Thus again  $\sign_{w_1}(w_2)=-\sign_{w_2}(w_1)$.
\end{proof}

Lemma \ref{lem:unisign} justifies Item (\ref{cp}) of the following definition,  showing that the choice of the tree vertex $w$ closest to either $u$ or $v$ ytields the same order. 

\begin{definition}\label{rposet}
Consider adjacent vertices $u,v \in R$.
\begin{enumerate}
\item\label{inc} If $\lab(u)=n$ and $\lab(v)=n+1$, then let $w \in R_0$ be the tree vertex nearest to $u$ (or $v$), and define: \\
\begin{tabular}{lcl} 
$ u < v$ & if & $\sign_w(u)= +1$, \\
$ u > v$ & if & $\sign_w(u)= -1$.\\
\end{tabular}
\item\label{cp} If $\lab(u)=\lab(v)$, then let $w \in R_0$ be the tree vertex nearest to $u$ (or $v$), and define: \\
\begin{tabular}{lcl} 
$ u > v$ & if & $\sign_w(u)= +1$, \\
$ u < v$ & if & $\sign_w(u)= -1$.\\
\end{tabular}
\end{enumerate}
\end{definition}

This provides the partial order we need on $R$.

\begin{observation}
The transitive closure of the the above relations makes $(R,<)$ a partial order.
\end{observation}

Now since  a graph embedding of $(R,r)$ is a similarity by Lemma \ref{similarityemb}, in particular it preserves the $\sign$ function at tree vertices, and as result is an order preserving embedding as defined above. Conversely we claim that an order embedding $\phi$ of $(R,<)$ is a graph embedding. 
First consider an edge $uv$ with $\lab(v)=\lab(u)+1$, and therefore $u$ and $v$ are comparable.  Then $u'=\phi(u)$ and $v'=\phi(v)$ have the same labels as $u$ and $v$ respectively due to the gadgets. But if $v''$ is the neighbour of $u'$ with $\lab(v'')=\lab(v)$, then any non-trivial path from $v'$ to $v''$ would contain a consecutive pair, and thus $v' \neq v''$ would imply that  $u'$ is incomparable to $v'$.
Similarly if $uv$ is a consecutive pair and thus again  $u$ and $v$ must be comparable. Let $w_1$ and $w_2 $ be the tree vertices closest to $u$ and $v$ respectively. From the above argument the paths $P_{w_1,u}$ and $P_{w_2,v}$ are mapped to the paths $P_{\phi(w_1),\phi(u)}$ and
$P_{\phi(w_2),\phi(v)}$ respectively. Since $\phi(w_1) \neq \phi(w_2)$, the path from $\phi(u)$ to $\phi(v)$ must contain a consecutive pair, 
and $\phi(u)$ is incomparable to $\phi(v)$ unless it is the edge $\phi(u)\phi(v)$ we are looking for.

We record this discussion as follows. 

\begin{lemma}\label{lem:graphposetmonoidR}
The monoid of embeddings of $(R,r)$ as a tree is the same as that of $(R,<)$ as a partial order. 
\end{lemma}
    
\subsection{Partial ordering on double rays}
This is where we use the spacial double rays $\SD'$ described in Figure \ref{doubleraysposets}. We  order such a double ray as a fence, and together with the ordering of gadgets above yields a partial ordering of double rays. 
 
 \begin{figure}[ht]
\includegraphics[width=.9\textwidth]{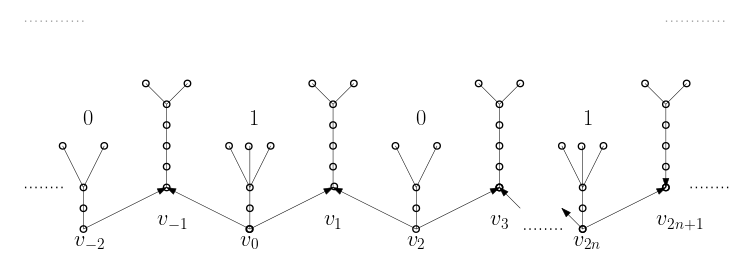}
\caption{The partial order on a typical double ray $\SD'$ (plus the order on the gadgets as above).}
\end{figure}
   
 Again we record the correspondence between graph and order embeddings. 

\begin{lemma}\label{lem:graphposetmonoidD}
 The monoid of embeddings of a double ray  $\SD'$ as a tree is the same as that of $(\SD',<)$ as a partial order. 
\end{lemma}   
        
\subsection{The posets $\langle \SP_s: s < \sss \rangle$}
For a non-zero $\sss \in \N$, the posets $\langle \SP_s: s < \sss \rangle$  we are seeking to produce consist of the trees $\langle \ST'_s: s < \sss \rangle$ constructed as before but now using the special double rays $\SD'_s$ (with type assignments on even indexed vertices), equipped with the (transitive closure) of the above orderings on copies of $(R,r)$ and double rays. 

It is now clear that order embeddings of $\SP_s$ will preserve copies of $(R,r)$ and double rays. First note that two vertices in a copy of $(R,r)$ are connected through a path in the comparability graph of $\SP_s$, hence so is their image; but if their image belongs to different copies of $(R,r)$  that path would go through an odd-indexed vertex of a double ray, which is impossible due to the gadgets being mutually non-embedable. Similarly two vertices on a double ray cannot be mapped to different double rays since again  the image of their comparability path would be required to contain a vertex of positive label, which is again impossible due to the gadgets being mutually non-embedable.

Hence order embeddings of $\SP_s$ coincide with graph embeddings of $\ST'_s$.  Thus $\SP=\SP_{0}$ has indeed exactly $\sss$ siblings up to isomorphism. This completes the proof of the second main theorem.

\end{document}